\documentclass[a4paper]{amsart}
\usepackage[T1]{fontenc}
\usepackage[utf8x]{inputenc}
\usepackage[english]{babel}
\usepackage{amsmath,amsfonts,amssymb}
\usepackage{graphicx}
\usepackage{hyperref}
\usepackage{enumerate}
\usepackage{xcolor}

\usepackage{amsthm}

\theoremstyle{definition}
\newtheorem{D}{Definition}[section]

\theoremstyle{plain}
\newtheorem*{theorem*}{Theorem}
\newtheorem{theorem}{Theorem}[section]
\newtheorem{lemma}[theorem]{Lemma}
\newtheorem{cor}[theorem]{Corollary}
\newtheorem{prop}[theorem]{Proposition}

\theoremstyle{definition}
\newtheorem{rem}[theorem]{Remark}

\newcommand{\R}{\ensuremath{\mathbb R}}
\newcommand{\N}{\ensuremath{\mathbb N}}

\newcommand{\s}{\ensuremath{\mathbb S}}
\newcommand{\eps}{\ensuremath{\varepsilon}}
\newcommand{\p}{\ensuremath{\frac{2n}{n-2}}}
\newcommand{\dist}{\ensuremath{\mbox{d}_g}}
\newcommand{\distArg}[1]{\ensuremath{\mbox{d}_{#1}}}
\newcommand{\norm}[1]{\ensuremath{\left\|#1\right\|}}
\newcommand{\cutoff}{\ensuremath{\rho_{\varepsilon}}}

\newcommand{\linfty}[1]{\ensuremath{(#1)}}
\newcommand{\tub}[1]{\ensuremath{\Sigma^{#1}}}
\newcommand{\Sob}[1]{\ensuremath{W^{1,#1}(X)}}

\DeclareMathOperator{\vol}{Vol}

\makeatletter
          \def\@setcopyright{}
          \def\serieslogo@{}
          \makeatother

\author{Ilaria Mondello}
\address{UPMC Université Paris 6, Institut de Mathématiques de Jussieu-Paris Rive Gauche, 
UMR 7586}
\email{ilaria.mondello@imj-prg.fr}
\thanks{This work is supported by a public grant overseen by the French National Research Agency (ANR) as part of the ``Investissements d'Avenir'' program (reference: ANR-10-LABX-0098)}
\title{An Obata singular theorem for stratified spaces}

\date{}
\begin{document}

\nocite{*}
\bibliographystyle{amsalpha} 

\maketitle

\begin{abstract}
Consider a stratified space with a positive Ricci lower bound on the regular set and no cone angle larger than $2\pi$. For such stratified space we know that the first non-zero eigenvalue of the Laplacian is larger than or equal to the dimension. We prove here an Obata rigidity result when the equality is attained: the lower bound of the spectrum is attained if and only if the stratified space is isometric to a spherical suspension. Moreover, we show that the diameter is at most equal to $\pi$, and it is equivalent for the diameter to be equal to $\pi$ and for the first non-zero eigenvalue of the Laplacian to be equal to the dimension. We finally give a consequence of these results related to the Yamabe problem. Consider an Einstein stratified space without cone angles larger than $2\pi$: if there is a metric conformal to the Einstein metric and with constant scalar curvature, then it is an Einstein metric as well. Furthermore, if its conformal factor is not a constant, then the space is isometric to a spherical suspension.  
\end{abstract}

\section*{Introduction}

The interest in the geometric study of singular metric spaces has been constantly increasing in the last years. Singular metric spaces appear easily as quotients or Gromov-Hausdorff limits of smooth manifolds. Thanks to the works of D.~Bakry and M.~Émery, or K.T.~Sturm, J.~Lott and C.~Villani, and many others, there are various way of defining the notions of curvature and dimension in a more general setting than the one of Riemannian manifolds. Some of the possible questions in this wide domain of mathematics can be collected in the following: which classical results of Riemannian geometry hold in the more general setting of singular metric spaces? 

In this paper we are interested in a particular class of singular metric spaces, which are called stratified spaces and generalize the notion of conical singularity. In fact, a compact stratified space $X$ can be decomposed into a regular dense set $\Omega$, which is a smooth manifold of dimension $n$, and in a singular set, with different components $\Sigma^j$ of possibly different dimensions $j$ smaller than $n$, called singular strata, with a local ``cone-like'' structure. What we mean is that the neighbourhood of a point in a singular stratum $\Sigma^j$ is the product of an Euclidean ball of dimension $j$ and a cone over a \emph{link}. This latter can be a compact manifold (in which case we have a manifold with simple edges) or a compact stratified space. Singular strata of codimension one are not admitted in the definition. The easiest examples of stratified space are manifolds with isolated conical singularities; in order to fix the ideas, one can also imagine to construct singularities along a curve, in which case the neighbourhood of a singular point is the product between an interval and a cone of the appropriate dimension. We observe that the link of a singular stratum of codimension two $\Sigma^{n-2}$ is a circle $\s^1$, and a cone over a circle has an angle $\alpha$: if $\alpha$ is smaller than $2\pi$, then the cone has positive curvature in the sense of Alexandrov, negative otherwise. We refer to $\alpha$ has the cone angle of the stratum $\Sigma^{n-2}$. On a stratified space we can consider an iterate edge metric, as defined in \cite{ALMP} or \cite{ACM12}, which is a smooth Riemannian metric on the regular set $\Omega$, and define the usual tools of geometric analysis. 

In \cite{Mondello} we introduced a class of stratified spaces, \emph{admissible} stratified spaces, which have, roughly speaking, a positive Ricci lower bound. What we mean is that the Ricci tensor is bounded by below by a positive constant in the regular set and there is an additional condition on the stratum of codimension two, in order to avoid the situation of a cone angle larger than $2\pi$, which would introduce in some sense negative curvature. The question is whether we can find geometric results on this class of stratified spaces which recover classical theorems for compact Riemannian manifolds with a positive Ricci lower bound. In \cite{Mondello} we already proved a singular version of the Lichnerowicz theorem: the first non-zero eigenvalue of the Laplacian is larger than or equal to the dimension of the space. Moreover, this allows one to deduce a Sobolev inequality with explicit constants depending only on the volume and on the dimension of the space. The main goal of this paper is to prove the following rigidity result for admissible stratified spaces:

\begin{theorem*}[Singular Obata]
Let $(X^n, g)$ be an admissible stratified space. The first non-zero eigenvalue of the Laplacian $\lambda_1(X)$ is equal to the dimension $n$ if and only if there exists an admissible stratified space $(\hat{X}^{n-1}, \hat{g})$ such that $(X^n,g)$ is isometric to the spherical suspension of $\hat{X}$, that is:
\begin{equation*}
\left(\hat{X} \times \left[-\frac{\pi}{2}, \frac{\pi}{2} \right], dt^2+\cos^2(t) \hat{g} \right). 
\end{equation*}
\end{theorem*}

When $(X^n,g)$ is a compact smooth manifold, the spherical suspension is simply a sphere of dimension $n$ with the canonical metric, and thus our theorem recover the known result of M.~Obata for compact smooth manifolds. Before proving the previous theorem, we recall a result due to D.~Bakry and M.~Ledoux (\cite{BakryLedoux}, Theorem 4) to deduce an upper bound on the diameter of an admissible stratified space: $\mbox{diam}(X)$ is less than or equal to $\pi$. The proof by D.~Bakry and M.~Ledoux relies on a spectral gap and on a Sobolev inequality as the ones we proved in \cite{Mondello}, therefore it is easily adaptable to our setting. Furthermore Theorem 4 \cite{BakryLedoux} shows that if the upper bound for the diameter is attained, then the first non-zero eigenvalue of the Laplacian is equal to $n$ and we know an explicit eigenfunction depending on the distance from a point. We prove that, in turns, if $\lambda_1(X)$ is equal to the dimension, then the diameter is equal to $\pi$. We have then the following theorem:

\begin{theorem*}[Singular Myers]
Let $(X,g)$ be an admissible stratified space of dimension $n$. Then the following statements are equivalent:
\begin{itemize}
\item[(i)] The first non-zero eigenvalue of the Laplacian $\Delta_g$ is equal to $n$.
\item[(ii)] The diameter of $X$ is equal to $\pi$.
\item[(iii)] There exist extremal functions for the Sobolev inequality. 
\end{itemize} 
\end{theorem*}

These results, together with a study of minimizing geodesics and tangent cones in an admissible stratified space, give us the main ingredients to prove the theorem ``à la'' Obata. 

We finally discuss an application of the rigidity result to the Yamabe problem, which consists in looking for a metric of constant scalar curvature among the conformal class of a given metric.  We refer to \cite{LeeParker} for a description of the Yamabe problem on compact smooth manifolds, and to \cite{ACM12} for the same in the setting of stratified spaces. The Yamabe problem has a variational formulation depending on a conformal invariant, called the Yamabe constant: this latter is defined as the infimum of the integral of the scalar curvature among conformal metrics of volume one. If there exists a conformal metric attaining the Yamabe constant, it has constant scalar curvature and it is called a Yamabe metric. A metric of constant scalar curvature is not necessarily a Yamabe metric, but we have shown in \cite{Mondello} that an Einstein metric on an admissible stratified space is a Yamabe metric. Here we give another proof of this result, under the assumption that a Yamabe minimizer exists. Moreover we show the following: 

\begin{theorem*}
Let $(X^n,g)$ be an admissible stratified space with Einstein metric. If there exists $\tilde{g}$ in the conformal class of $g$, not homothetic to $g$, with constant scalar curvature, then $\tilde{g}$ is an Einstein metric as well and $(X^n,g)$ is isometric to the spherical suspension of an admissible stratified space $(\hat{X}^{n-1}, \hat{g})$ with Einstein metric.
\end{theorem*}

This is also true for compact smooth manifolds due to another theorem of M.~Obata. 

We notice that a Myers theorem has been proven by C.~Ketterer in \cite{Ketterer2} for metric measure spaces which satisfy a curvature-dimension condition $RCD(K,n)$. Moreover, if the upper bound is attained, then the metric measure space is isometric to a spherical suspension. His proof relies on a splitting theorem of N.~Gigli \cite{Gigli}. As a consequence, the author also proved an Obata rigidity theorem in \cite{Ketterer}. Our analogous result clearly applies in a less general setting, but the advantage if its proof is that it is based on simple tools coming from Riemannian geometry, and essentially on the study of an equation for the Hessian of a function. It remains an interesting question whether admissible stratified spaces satisfy a curvature-dimension condition in the sense of Bakry-Émery, Sturm-Lott-Villani or $RCD(K,n)$, since they could give new concrete examples of metric measure spaces belonging to this setting. \newline

\textbf{Acknowledgements}: I would like to thank Gilles Carron for countless discussions, Rafe Mazzeo for helpful suggestions, Kazuo Akutagawa, Gilles Courtois and Vincent Minerbe for their remarks when I was completing this work. 

\section{Preliminaries}

We introduce here a detailed definition of a stratified space. For this purpose, we precise that for a truncated cone $C(Z)$ over a compact metric space $Z$ we mean the product $Z\times [0,1]$ with the equivalence relation $(z_1, 0)\sim (z_2,0)$ for all $z_1, z_2$ in $Z$: we identify all the points in $Z\times \{0\}$ to a unique point, called the vertex of the cone. We say that a truncated cone is of size $\delta$ if we consider the interval $[0, \delta]$ instead of $[0,1]$. If $Z$ is a compact manifold endowed with the Riemannian metric $k$, then a conic metric on $C(Z)$ has the form $dr^2+r^2k$. 

\begin{D}
Let $(X,d)$ be a compact metric space. We say that $X$ is a stratified space if it admits a decomposition of the form:
\begin{equation*}
X=\Omega \sqcup \Sigma, 
\end{equation*}
where $\Omega$ is an open smooth manifold of dimension $n$ dense in $X$, and $\Sigma$ is the disjoint union of a finite number $N$ of components $\Sigma^j$, $j=1, \ldots N$, called singular strata, which are smooth manifolds of dimension $j$. The stratum of dimension $(n-1)$ is empty. 

For each $\Sigma^j$ there exist a neighbourhood $\mathcal{U}_j$ of $\Sigma^j$, a retraction $\pi_j$, a radial function $\rho_j$: 
\begin{equation*}
\pi_j: \mathcal{U}_j \rightarrow \Sigma^j, \quad \rho_j: \mathcal{U}_j \rightarrow [0,1]
\end{equation*}
and a stratified space $Z_j$ such that $\pi_j$ is a cone bundle, whose fibre is a truncated cone over $Z_j$. The stratified space $Z_j$ is called the \emph{link} of the stratum $\Sigma^j$. 
\end{D} 

In the following we will refer to $\Omega$ and $\Sigma$ respectively as the regular and the singular set of $X$. We can reformulate the condition on the strata $\Sigma^j$ by saying that for each point $x$ in $\Sigma^j$ there exist a neighbourhood $\mathcal{W}_x$, a positive radius $\delta_x$ and a homeomorphism $\varphi_x$ between $\mathcal{W}_x$ and a product between an Euclidean ball $\mathbb{B}(\delta_x)$ and a truncated cone $C_{\delta_x}(Z_j)$ of size $\delta_x$ over the link $Z_j$. Moreover, $\varphi_x$ is a diffeomorphism between the regular part of $\mathcal{W}_x$ and $\mathbb{B}(\delta_x)\times C_{\delta_x}(Z_j^{\mbox{\tiny{reg}}}) \setminus  \mathbb{B}(\delta_x)\times \{0 \}$. In the rest of the paper we will treat $\varphi_x$ as an identification between $\mathcal{W}_x$ and the product $\mathbb{B}^j(\delta_x)\times C_{\delta_x}(Z_j)$. \newline

One can define an admissible metric on a stratified space: for a precise discussion we refer to section 3 of \cite{ALMP} and section 2.1 of \cite{ACM12}. For the purposes of this paper, the reader only needs to know that an admissible metric $g$ is a smooth Riemannian metric on the regular set $\Omega$ and near to the stratum $\Sigma^j$ it is a perturbation of the model metric $g_0=\xi_j+ dr^2+r^2k_j$, where $\xi_j$ is the Euclidean metric on $\R^j$ and $k_j$ is an iterated edge metric on the link $Z_j$. More precisely, if $x$ belongs to $\Sigma^j$ and $\mathcal{W}_x$, $\delta_x$ and $\varphi_x$ are defined as above, we have for any $r < \delta_x$:
\begin{equation*}
|\varphi^{*}_x g - g_0| \leq \Lambda r^{\alpha} \quad \mbox{ on } \mathbb{B}^j(r)\times C_r(Z_j). 
\end{equation*}
where $\Lambda$ is a positive constant and $\alpha >0$ does not depend on $j$. \newline 

In the following we will consider minimizing geodesics, that in this context are Lipschitz curves which minimize the distance between two points. We will need to use the uniqueness of a minimizing geodesic starting from a regular point with fixed speed: for this to be true, the metric must be $C^2$. We then assume that near each stratum $\Sigma^j$ the perturbation of the model metric $\varphi^*_x g -g_0$ has coefficients in $C^2$, and that the same is true for the metric $k_j$ on the links. \newline

On a stratified space it is possible to define the usual analytic tools of geometric analysis. We are mostly interested in the Sobolev space $W^{1,2}(X)$ and in the Laplacian operator. The first one is defined as the closure of the Lipschitz function on $X$ with respect to the usual norm:
\begin{equation*}
\norm{f}_{1,2}^2= \norm{f}_2^2+\norm{df}_2^2. 
\end{equation*}
Thanks to the assumption that the codimension one stratum does not exist, the smooth functions with compact support in the regular set $\Omega$ are dense in $W^{1,2}(X)$. A standard proof of this can be found in \cite{Thesis}. In \cite{ACM12}, the usual Sobolev embeddings which hold on compact Riemannian manifolds are proven in the setting of stratified spaces as well. In particular we have the following Sobolev inequality: there exist positive constants $A$ and $B$ such that for any $u$ in $W^{1,2}(X)$
\begin{equation*}
\norm{u}_{\p}^2\leq A\norm{du}_2^2+ B\norm{u}_2^2. 
\end{equation*}
The Laplacian operator $\Delta_g$ is the positive self-adjoint operator defined as the Friedrichs extension of the semi-bounded Dirichlet quadratic form $\mathcal{E}$:
\begin{equation*}
\mathcal{E}(u) =\int_X |du|^2 dv_g. 
\end{equation*}
defined for $u$ in $C^{\infty}_0(\Omega)$. 

\subsection*{Tangent cones and geodesic balls}

It will be useful to introduce another description for a neighbourhood of a singular point, which relays on the notion of tangent sphere. First, for each point in a stratified space, the pointed Gromov-Hausdorff limit of $(X, \lambda d, x)$ as $\lambda$ tends to infinity exists, it is unique and it is carries an exact cone metric. We refer to this limit as the tangent cone at $x$. When $x$ is a point in $\Omega$, the tangent cone is simply the Euclidean space $\R^n$. If $x$ belongs to $\Sigma^j$, the tangent cone is a cone over the following  stratified space:
\begin{equation*}
S_x=\left[0,\frac{\pi}{2} \right]\times \s^{j-1} \times Z_j
\end{equation*}
endowed with the metric $h_x=d\theta^2+\cos^2\theta g_{\s^{j-1}}+\sin^2\theta k_j$. We refer to $S_x$ as the $(j-1)$-fold spherical suspension of the link $Z_j$, and more often as the tangent sphere at $x$. 

In \cite{ACM14}, the authors showed that for each singular point $x$ there exist a sufficiently small radius $\eps_x$, a constant $\kappa$ and an open neighbourhood $\Omega_x$ of $x$ such that the geodesic ball centred at $x$ is included in $\Omega_x$, $\Omega_x$ is homeomorphic to the cone $C_{\kappa \eps_x}(S_x)$ and moreover in $B(x,\eps_x)$ the metric $g$ differs from the exact cone metric $dr^2+r^2 h_x$ for:
\begin{equation*}
|g-(dr^2+r^2h_x)|\leq \Lambda \eps_x^{\alpha}. 
\end{equation*}
For a more detailed description of the previous, we refer to section 2.2 in \cite{ACM14} and to the first chapter of \cite{Thesis}. 

\subsection*{Admissible stratified spaces}

Most of the results of this paper are stated for a class of stratified spaces, called \emph{admissible} and introduced in \cite{Mondello}. We recall their definition:

\begin{D}
A stratified space $(X^n,g)$ is an admissible stratified space if it satisfies the following two conditions:
\begin{itemize}
\item[(i)] The Ricci tensor on $\Omega$ is such that $Ric_g\geq(n-1)g$.
\item[(ii)] The stratum $\Sigma^{n-2}$ of codimension 2, if it is not empty, has angle $\alpha$ strictly less than $2\pi$. 
\end{itemize}
\end{D}

The second condition is to exclude the situation of a cone of angle $\alpha > 2\pi$, which in some sense would introduce negative curvature, thus an obstruction to extend results holding on smooth manifolds with a positive Ricci lower bound. For admissible stratified spaces, we proved in \cite{Mondello} a singular version of the Lichnerowicz theorem: 

\begin{theorem}[Singular Lichnerowicz]
Let $(X^n, g)$ be an admissible stratified space. Then the first non-zero eigenvalue $\lambda_1(X)$ of the Laplacian $\Delta_g$ is larger than or equal to the dimension $n$. 
\end{theorem}

The proof of this theorem is by iteration on the dimension of the stratified space, and it consists in using Bochner-Lichnerowicz formula on the regular set and getting the suitable regularity on the eigenfunctions $\varphi$. Then by using the appropriate cut-off functions $\rho_{\eps}$, vanishing in a tubular neighbourhood of the singular set and being equal to one elsewhere, we obtain that for any eigenfunction $\varphi$ of the Laplacian associated to the eigenvalue $\lambda$ the following holds:
\begin{equation}
\label{LichIn}
\begin{split}
\left(1 -\frac{(n-1)}{\lambda} -\frac{1}{n}\right)\int_X \rho_{\eps}(\Delta_g \varphi)^2dv_g & \geq 
\left( 1 -\frac{(n-1)}{\lambda}\right)\int_X \rho_{\eps}(\Delta_g \varphi)^2dv_g \\
 & -\int_X \rho_{\eps}|\nabla d\varphi|^2 dv_g \geq 0.
\end{split} 
\end{equation}
Passing to the limit as $\eps$ goes to zero, this implies the desired inequality. Furthermore, the singular Lichnerowicz theorem has consequences on the regularity of the non-negative solutions to a Schrödinger equation of the form $\Delta_g u = F(u)$, where $F$ is locally Lipschitz. In particular, for an eigenfunction $\varphi$ we have that $\varphi$ belongs to $W^{2,2}(X)$ and its gradient is bounded on $X$ (see Corollary 2.12 in \cite{Thesis}). 

We observe that if there exists an eigenfunction $\varphi$ associated to the eigenvalue $n$, then the inequality $\eqref{LichIn}$ implies that its Hessian must satisfy $|\nabla d\varphi|^2=(\Delta_g\varphi)^2/n$ on the regular set. Therefore we are in the case of equality in the Cauchy-Schwarz inequality and we get that the Hessian of $\varphi$ is proportional to the metric $g$ in the regular set of $X$: 
\begin{equation}
\label{hessian}
\nabla d\varphi=-\varphi g \quad \mbox{ on } \Omega. 
\end{equation}
If the Hessian of a scalar function $\varphi$ satisfies an equation of the form $\nabla d\varphi= \rho g$ for some function $\rho$, then $\varphi$ is called in the literature a concircular scalar field. Its gradient $X=d\varphi$ is a conformal vector field, which means that the Lie derivative of the metric along $X$ is proportional to $g$. The existence of a concircular scalar field or of a conformal vector field on a compact, or complete, smooth manifold can lead to various consequences. For example, Y.~Tashiro in \cite{Tashiro} classified complete manifolds possessing a concircular scalar field. See also Sections 2 and 3 of \cite{Montiel} for a brief but complete presentation of some known results about the subject. 

In the setting of admissible stratified spaces as well, the equation $\eqref{hessian}$ is a key point in proving a rigidity result, as it will be clear in the proofs of Theorem \ref{MyersSing} and \ref{ObataSing}. 

\begin{rem}
\label{rem}
If $(X^n,g)$ is an admissible stratified space, then each of its links $Z_j$ and the tangent sphere at each point $S_x$ are admissible stratified spaces as well (see Lemma 1.1 in \cite{Mondello}). As a consequence of this and of the singular Lichnerowicz theorem, the first non-zero eigenvalue of the Laplacian on each tangent sphere is larger than $(n-1)$. 
\end{rem}

As we recalled in the introduction, the singular Lichnerowicz theorem allows one to prove that a Sobolev inequality with explicit constants holds on an admissible stratified space:

\begin{theorem}[Sobolev inequality]
\label{BakryS}
Let $X$ be an admissible stratified space of dimension $n$. Then for any $\displaystyle 1<p\leq 2n/(n-2)$ a Sobolev inequality of the following form holds:
\begin{equation}
\label{SobP}
V^{1-\frac{2}{p}}\norm{f}_p^2\leq \norm{f}_2^2+\frac{p-2}{n}\norm{df}_2^2.
\end{equation}
where $V$ is the volume of $X$ with respect to the metric $g$.
\end{theorem} 

A Sobolev inequality of the previous form was proven by S.~Ilias in \cite{Ilias} for compact smooth manifolds with Ricci tensor bounded by below by a positive constant, and by D.~Bakry in \cite{Bakry} for a much general setting. Our proof is inspired by the argument due to D.~Bakry. 

We now dispose of all the necessary tools to prove the upper bound on the diameter of an admissible stratified space. 

\section{A Myers singular theorem}

A classical result holding for smooth Riemannian manifolds is the Myers theorem: if $(M^n,g)$ is complete, connected, and its Ricci tensor is bounded by below by $(n-1)g$, then the diameter of $M$ is less or equal than $\pi$. In \cite{BakryLedoux}, the authors have proven that this kind of lower bound can be shown in a great generality, on a probability measure space with a Markov generator which satisfies a curvature-dimension condition. Moreover, the proof relies only on analytical tools, in particular on the existence of a Sobolev inequality of the form $\eqref{SobP}$ and on the choice of the appropriate test functions (see Section 2 in \cite{BakryLedoux} for the details). 
The previous theorem gives us the Sobolev inequality needed to apply D.~Bakry and M.~Ledoux's proof. As a consequence, the Myers theorem holds on admissible stratified spaces in the following sense:

\begin{theorem}[Singular Myers Theorem]
\label{MyersSing}
Let $(X,g)$ an admissible stratified space. Let us define its Lipschitz diameter as:
\begin{equation*}
\mbox{\emph{diam}}_L(X)=\sup\left\{ ||\tilde{f}||_{L^{\infty}(X \times X)}; f \in \mbox{\emph{Lip}}_1(X)\right\}
\end{equation*}
where $\tilde{f}(x,y)=f(x)-f(y)$ and  $\mbox{\emph{Lip}}_1(X)$ is the set of Lipschitz functions with Lipschitz constant less or equal than one. Then $\mbox{\emph{diam}}_L(X)$ is less or equal than $\pi$. 
\end{theorem}

Observe that on a smooth Riemannian manifold, what we called Lipschitz diameter coincides with the usual diameter associated to the Riemannian metric. We remark that it is possible to prove the following lemma:

\begin{lemma}
\label{curve}
Let $(X,g)$ be a stratified space of dimension $n$ and let \mbox{$\gamma: [0,1] \rightarrow X$} be a Lipschitz curve in $X$. Let $L_g(\gamma)$ denote its length. For any $\eps>0$ there exists a curve $\gamma_{\eps}$ such that $\gamma_{\eps}((0,1))$ is contained in the regular set $\Omega$ and $L_g(\gamma_{\eps}) \leq (1+\eps)L_g(\gamma)$. 
\end{lemma}

This implies two facts: first, a function $u$ in $C^1(\Omega)\cap C^0(X)$ whose gradient is bounded in $L^{\infty}(X)$ by a constant $c$ is a Lipschitz function on the whole of $X$, with Lipschitz constant less or equal than $c$; moreover, the Lipschitz diameter coincides with the diameter associated to the metric $g$, and we can avoid any distinction between the two. 

We are going to show that an admissible stratified space has diameter equal to $\pi$ if and only if the first non-zero eigenvalue of the Laplacian is equal to the dimension of the space. Thanks to Theorem 4 in \cite{BakryLedoux} this is in turn equivalent to the existence of extremal functions for the Sobolev inequality $\eqref{SobP}$ which only depend on the distance from a point. 

\begin{theorem}
Let $(X,g)$ be an admissible stratified space of dimension $n$. Then the following statements are equivalent:
\begin{itemize}
\item[(i)] The first non-zero eigenvalue of the Laplacian $\Delta_g$ is equal to $n$.
\item[(ii)] The diameter of $X$ is equal to $\pi$.
\item[(iii)] There exist extremal functions for the Sobolev inequality. 
\end{itemize} 
\end{theorem}

\begin{proof}
If the diameter of $X$ is equal to $\pi$, then its Lipschitz diameter is equal to $\pi$, and then Theorem 4 in \cite{BakryLedoux} implies both the existence of functions attaining the equality in Sobolev inequality and of an eigenfunction associated to the eigenvalue $n$. In particular, if $P$ is a point in $X$ with and antipodal point $N$, $\mbox{d}_g(P,N)=\pi$, then the function $\varphi(x)=\cos(\mbox{d}_g(P,x))$ is such that $\Delta_g \varphi= n \varphi$. 

As a consequence, we have to prove that if the first non-zero eigenvalue of the Laplacian is equal to the dimension of the space, then its diameter is equal to $\pi$. If we find a Lipschitz function $f$  which takes values in an interval of length $\pi$ and whose Lipschitz constant is smaller or equal than one, then we have that $\mbox{diam}_L(X)=\pi$, and thanks to the previous lemma we get the desired value for the diameter.   

Consider $\varphi$ an eigenfunction associated to the eigenvalue $n$: as we recalled above, its gradient belongs to $W^{1,2}(X)$ and it is bounded. Moreover, its Hessian is proportional to the metric $g$ on the regular set $\Omega$, since $\varphi$ satisfies the equation $\eqref{hessian}$. As a consequence, we can show that the quantity $|\nabla \varphi|^2+\varphi^2$ is a constant on the regular set $\Omega$. In fact we have:
\begin{equation*}
d(|\nabla \varphi|^2+\varphi^2)=2\varphi d\varphi + 2\nabla d\varphi( \cdot, \nabla \varphi)=2\varphi d\varphi - 2\varphi d\varphi=0.
\end{equation*}
Then, up to multiplying by a constant, we can assume without loss of generality that:
\begin{equation}
\label{UguaglianzaRisolutiva}
|\nabla \varphi|^2+\varphi^2=1 \quad \mbox{ on } \Omega.
\end{equation}
This equality tells us that $\varphi$ takes values between $-1$ and $1$. Let us consider the function $f$ defined as follows:
\begin{equation*}
f=\arcsin(\varphi).
\end{equation*}
Its gradient is bounded on the regular set $\Omega$, because the gradient of $\varphi$ belongs to $L^{\infty}(X)$, and then $f$ belongs to $\mbox{Lip}(X)$ as well. Moreover, by definition $\nabla f$ has norm equal to one at each regular point: thanks to Lemma $\ref{curve}$ this implies that the Lipschitz constant of $f$ on the whole $X$ is less or equal than one. In order to conclude, we need to show that the image of $X$ by $f$ is equal to $[-\pi/2, \pi/2]$. This is clearly equivalent to proving that $\varphi$ has the closed interval $[-1,1]$ as image. 

Let us define $\mathcal{U}_+$ as the set on which $\varphi$ is strictly positive. Observe that $\Omega \cap \mathcal{U}_+$ is not empty, since $\varphi$ changes sign on $X$, and $\Omega$ is dense in $X$. Moreover $\Omega \cap \mathcal{U}_+$ is dense in $\mathcal{U}_+$, since $\Omega$ is dense and $\mathcal{U}_+$ is an open set in $X$.

Consider and the following problem with Dirichlet condition at the boundary: 
\begin{equation*}
\begin{cases}
\Delta_g f=\lambda f \mbox{ in } \mathcal{U}_+\\
f=0 \mbox{ on } \partial\mathcal{U}_+. 
\end{cases}
\end{equation*}
This problem has a variational formulation: we can define the first non-zero Dirichlet eigenvalue on $\mathcal{U}_+$ as the infimum of the Dirichlet energy on functions in $W^{1,2}_0(\mathcal{U}_+)$, that is:
\begin{equation*}
\lambda_1(\mathcal{U}_+)= \inf\left\{\mathcal{E}(\psi)= \frac{\norm{d\psi}^2_2}{\norm{\psi}^2_2}, \psi\in W^{1,2}_0(\mathcal{U}_+) \right\}
\end{equation*}

Assume by contradiction that the maximum of $\varphi$ is equal to $M$, strictly smaller than $1$. We state that this implies the existence of a function $u: [0, M] \rightarrow \R_{+}$ such that $u(0)=0$ and 
\begin{equation*}
\Delta_g (u\circ \varphi) = n\varphi u'(\varphi) -(1-\varphi^2)u''(\varphi) > n(u \circ \varphi), \mbox{ on  } \Omega \cap \mathcal{U}_+. 
\end{equation*} 
This means that we can find a function $u$ which vanishes at 0, is positive on $(0,M]$ and satisfies the following differential inequality on $(0,M]$:
\begin{equation}
\label{wish}
-u''(t)(1-t^2)+ ntu'(t)>nu(t).
\end{equation}
Let $\alpha >1$, to be chosen later, and consider $u_{\alpha}(t)=t-t^{\alpha}$. By replacing in the differential inequality, we reformulate $\eqref{wish}$ in the following way:
\begin{align*}
\alpha(\alpha-1)t^{\alpha -2}(1-t^2)+ nt(1-\alpha t^{\alpha-1}) &> n(t-t^{\alpha}). \\
\alpha (\alpha-1)t^{\alpha -2} -\alpha(\alpha-1)t^{\alpha}-n\alpha t^{\alpha} +n t^{\alpha}&>0 \\
\alpha (\alpha -1) t^{\alpha-2} -(\alpha-1)t^{\alpha}(\alpha+n) &>0. 
\end{align*}
Now by multiplying by $(\alpha-1)t^{2-\alpha}>0$ we get:
\begin{equation*}
\alpha -t^2(\alpha+n)>0. 
\end{equation*}
Therefore the question becomes to find an $\alpha>1$ such that the previous inequality is satisfied. 
The second degree polynomial appearing in the left-hand side of the previous inequality has a solution in $[0,1]$ at $t_0(\alpha)=\sqrt{\alpha(\alpha+n)^{-1}}$, and it is positive between $0$ and $t_0(\alpha)$. Since this last quantity tends to one as $\alpha$ goes to infinity, and since $M$ is strictly smaller than one, we can choose $\alpha$ large enough so that $t_0(\alpha)$ is strictly larger than $M$. For such $\alpha$ the function $u_{\alpha}$ satisfies the desired differential inequality, it is positive in $(0, M]$ and vanishes at $0$. From now on we denote $u_{\alpha}$ simply by $u$, and $u\circ \varphi$ by $\phi$. 

Let $\eps$ be a positive real number and define $u_{\eps}=u+\eps$: then $u_{\eps}$ is strictly positive and, if we consider $\phi_{\eps}=u_{\eps}\circ \varphi$, the Laplacian of $\phi_{\eps}$ satisfies $\Delta_g \phi_{\eps} > n \phi$ on $\Omega \cap \mathcal{U}_+$.

For any positive function $\psi$ belonging to $W^{1,2}_0(\mathcal{U}_+)$ we can define $v= \psi/{\phi_{\eps}}$, which still belongs to $W^{1,2}_0(\mathcal{U}_+)$. By integration by parts and using that $\Omega \cap \mathcal{U}_+$ is dense in $\mathcal{U}_+$ we obtain:
\begin{align*}
\int_{\mathcal{U}_+}|d\psi|^2dv_g 
&=\int_{\mathcal{U}_+}|d(v\phi_{\eps})|^2dv_g = \int_{\mathcal{U}_+}(v^2|d\phi_{\eps}|^2+2v\phi_{\eps}(dv, d\phi_{\eps})_g+\phi_{\eps}^2|dv|^2)dv_g \\
& \geq \int_{\mathcal{U}_+}(v^2|d\phi_{\eps}|^2+2v\phi_{\eps}(dv, d\phi_{\eps})_g)dv_g = \int_{\mathcal{U}_+} (d(v^2\phi_{\eps}), d\phi_{\eps})_g dv_g= \\
& = \int_{\mathcal{U}_+}\phi_{\eps} v^2 \Delta_g \phi_{\eps} dv_g = \int_{\mathcal{U}_+ \cap \Omega} \phi_{\eps} v^2 \Delta_g \phi_{\eps} dv_g.
\end{align*}
Now, by using that $\Delta_g \phi_{\eps}>n\phi$ on $\mathcal{U}_+ \cap \Omega$ in the last integral, and since by definition $v=\psi/\phi_{\eps}$ we get:
\begin{equation*}
\int_{\mathcal{U}_+}|d\psi|^2dv_g > n \int_{\mathcal{U}_+ \cap \Omega} \phi_{\eps} \phi v^2 dv_g 
= n \int_{\mathcal{U}_+ \cap \Omega}\psi^2 \frac{\phi}{\phi_{\eps}}dv_g = n \int_{\mathcal{U}_+ } \psi^2 \frac{\phi}{\phi_{\eps}}dv_g. 
\end{equation*}
Now observe that $\phi/ \phi_{\eps}$ is smaller than one, it converges to one almost everywhere when $\eps$ goes to zero, and when we pass to the limit, by the dominated convergence theorem, we get:
\begin{equation*}
\int_{\mathcal{U}_+} |d\psi|^2dv_g \geq n \int_{\mathcal{U}_+}\psi^2 dv_g.
\end{equation*}
This shows that $\lambda_1(\mathcal{U}_+)$ is larger than or equal to $n$. 

The eigenfunction $\varphi$ associated to $n$ is a positive function on $\mathcal{U}_+$ belonging to $W^{1,2}_0(\mathcal{U}_+)$, and therefore $\lambda_1(\mathcal{U}_+)$ is equal to $n$. Moreover, we can apply the same calculations as above with $\psi=\varphi$. We can write $\varphi$ as $v\phi$, where $v$ is strictly positive on $\mathcal{U}_+$ and it is defined by $v= (1-\varphi^{\alpha -1})^{-1}$, since by definition $\phi = \varphi-\varphi^{\alpha}$. We can easily deduce that $v$ must be a positive constant. In fact we have:
\begin{equation*}
\begin{split}
n \int_{\mathcal{U}_+} \varphi^2 dv_g &= \int_{\mathcal{U}_+} |d\varphi|^2dv_g =\int_{\mathcal{U}_+} (\phi^2|dv|^2 + \phi v^2 \Delta_g \phi)dv_g \\
&>\int_{\mathcal{U}_+} \phi^2|dv|^2dv_g+ n\int_{\mathcal{U}_+} \varphi^2 dv_g. 
\end{split}
\end{equation*}
This means that $dv=0$, $v$ must be equal to a constant $c$ and $\phi$ is a multiple of $\varphi$, therefore an eigenfunction relative to $n$. This is a contradiction, since we have shown that $\Delta_g \phi$ is strictly larger than $n \phi$ on $\Omega \cap \mathcal{U}_+$. Therefore, the maximum of $\phi$ on $\mathcal{U}_+$ must be equal to one. 

Remark that, in particular, we have proven that the Dirichlet problem on $\mathcal{U}_+$ has a unique positive solution up to multiplication factors. 

Analogously, the minimum of $\varphi$ is equal to $-1$: therefore the image of $X$ via $\varphi$ is $[-1,1],$, and via $f$ is $[-\pi/2, \pi/2]$. Thanks to Theorem $\ref{MyersSing}$ we know that the Lipschitz diameter is less or equal than $\pi$, and then we get the equality, as we wished. 
\end{proof}

\section{Obata singular theorem}

In this section we are going to prove a rigidity result for an admissible stratified space such that the first non-zero eigenvalue of the Laplacian is equal to the dimension. This theorem recovers the one proved by M.~Obata \cite{Obata} for compact smooth manifolds $(M^n,g)$ with Ricci tensor bounded by below by $(n-1)g$. For an alternative discussion of the proof in the case of Riemannian manifolds we refer to Theorem D.I.6 in \cite{BGM}.  

\begin{theorem}[Singular Obata theorem]
\label{ObataSing}
Let $(X,g)$ an admissible stratified space of dimension $n$. The first eigenvalue of the Laplacian $\Delta_g$ is equal to $n$ if and only if there exists an admissible stratified space $(\hat{X},\hat{g})$ of dimension $(n-1)$ such that $(X,g)$ is isometric to the spherical suspension of $\hat{X}$:
\begin{equation}
S(\hat{X})=\left[-\frac{\pi}{2},\frac{\pi}{2} \right]\times \hat{X}.
\end{equation}
Endowed with the metric $dt^2+\cos^2(t)\hat{g}$. 
\end{theorem}

This theorem has an immediate consequence for cones over admissible stratified spaces whose diameter is equal to $\pi$, which is going to play a role in the proof. We are first going to prove the following: 

\begin{cor}[Splitting]
\label{splitting}
Let $(X^n,g)$ be an admissible stratified space of diameter equal to $\pi$. Then the cone $C(X)$ splits into the product $\R \times C(Y)$, where $(Y,k)$ is an admissible stratified space. 
\end{cor}

\begin{proof}
It is an easy fact that a cone over a stratified space $(X^n,g)$ splits a factor $\R$ if and only if $(X^n,g)$ is a spherical suspension over a stratified space $(Y, k)$. In fact, consider the metric $dr^2+ds^2+s^2k$ on $\R \times C(Y)$, with $r\in \R$ and $s \in \R^+$. We define the change of variables:
\begin{equation*}
r= \rho \sin(\theta), \quad s=\rho \cos(\theta) \quad  \mbox{ for } \theta \in \left( -\frac{\pi}{2}, \frac{\pi}{2}\right). 
\end{equation*}
Then replacing in the product metric we get:
\begin{equation*}
d\rho^2+ \rho^2(d\theta^2+\cos^2(\theta)k),
\end{equation*}
on the cone over the spherical suspension of $(Y,k)$.

Theorem \ref{ObataSing} states that an admissible stratified space $(X^n,g)$ of diameter $\pi$ is isometric to a spherical suspension over $(\hat{X}, \hat{g}$), and therefore, the cone over $(X^n,g)$ splits a factor $\R$. 
\end{proof}

\begin{rem}\label{m_split}
In the previous Corollary, if $(Y, k)$ has diameter equal to $\pi$, we can iterate this argument, until we get the splitting $\R^m \times C(Y_0)$ for $m \geq 1$ and an admissible stratified space $(Y_0,k_0)$ of diameter strictly less than $\pi$. 
\end{rem}

\begin{rem}\label{Busemann}
Under the assumption of the previous Corollary, let $P$ and $N$ two points in $X$ at distance $\pi$, which in the coordinates given by the spherical suspension corresponds to $\{-\pi/2 \}\times Y$ and $\{\pi/2\}\times Y$ respectively. Consider the geodesic $\gamma_0$ in $C(X)$ relying the vertex $0$ of the cone with $P$ and $N$. Since $C(X)$ is isometric to $\R \times C(Y)$ endowed with the metric $d\rho^2+\rho^2(d\theta^2+\cos^2(\theta)k$, the geodesic $\gamma_0$ is defined on the whole $R$: it is the radius connecting $0$ and $N$ on $\R_+$, the one connecting $0$ and $P$ on $\R_{-}$. We claim that the first coordinates $\rho$ in the metric corresponds to the opposite of the Busemann function of the geodesic $\gamma_0$. Indeed, let $x$ be a point in $C(X)=\R\times C(Y)$ of coordinates $(\rho(x), \theta(x), y)$ and $\gamma_0(t)=(t,0,0)$ a point of the geodesic $\gamma_0$. The Busemann function associated to $\gamma_0$ is defined as:
\begin{equation*}
B_{\gamma_0}(x)=\lim_{t\rightarrow +\infty}(\mbox{d}_{C(X)}(\gamma_0(t),x)-t),
\end{equation*}
and by using the formula for the distance in $C(X)=\R\times C(Y)$ we get:
\begin{align*}
B_{\gamma_0}(x)&=\lim_{t\rightarrow +\infty}(\sqrt{|t-\rho(x)|^2+s(x)^2}-t)\\
&= \lim_{t\rightarrow +\infty} \frac{-2\rho(x)t+\rho(x)^2+s(x)^2}{\sqrt{|t-\rho(x)|^2+s(x)^2}+t}=-\rho(x), 
\end{align*}
As we claimed above. Observe also that the Busemann function of $\gamma_0$ is onto on $\R$, since for any point $\gamma_0(s)$ of the geodesic we have $B_{\gamma_0}(\gamma_0(s))=-s$.  
\end{rem}

For the purposes of the proof of Theorem \ref{ObataSing}, we need some information about minimizing geodesics on an admissible stratified space. For a minimizing geodesic we mean a Lipschitz curve $\gamma: I \rightarrow X$ such that for any $t_1, t_2$ in the interval $I$ we have $\mbox{d}_{g}(\gamma(t_1), \gamma(t_2))=|t_2-t_1|$. We point out here that little is known about minimizing geodesics on general stratified spaces, their regularity and the uniqueness of a minimizing geodesic between two points, in particular when one or both of them belong to the singular set. 

\begin{lemma}
\label{Lemma1}
Let $X$ be an admissible stratified space, $x$ be in $X$ and $\gamma : [0,1] \rightarrow X$ a Lipschitz minimizing geodesic starting from $x$. Then $\dot{\gamma}(0)$ is well-defined and unique. 
\end{lemma}

\begin{proof}
We know that if $X$ is an admissible stratified space, the diameter of $X$ is smaller or equal than $\pi$, and moreover, thanks to Remark \ref{rem}, that each tangent sphere is an admissible stratified space: therefore, the diameter of each tangent sphere is less or equal than $\pi$. As a consequence, if we consider the tangent cone $C(S_x)$ the distance between two points $(t,y)$ and $(s,z)$ is given by:
\begin{equation}
\label{ConeDist}
\mbox{d}_{C}((t,y),(s,z))=\sqrt{t^2+s^2-2rs\cos \distArg{S_x}(y,z)}
\end{equation}
Recall that for $t$ small enough, a neighbourhood $B(x,t)$ of a point $x$ in $X$ in included in an open neighbourhood $\Omega_x$ of $x$ which is homeomorphic to a truncated cone of size $kt$, for a positive constant $k$, over the tangent sphere $S_x$ at $x$. Moreover, the metric $g$ on $B(x,t)$ and the conic metric on $C_{[0,kt]}(S_x)$ differ for an error which is proportional to $t^{\alpha}$ for $\alpha>0$. If we consider this estimate in terms of the distances associated to $g$ and to the conic metric, we get the following: for any $y$ in $B(x,t)$ with coordinates $(r,z)$ in $C_{[0, kt)}(S_x)$ we have
\begin{equation}
\label{distanceSommet}
|\dist(x,y)-\mbox{d}_{C}(0,(r,z))|\leq \Lambda t^{1+\alpha},
\end{equation}
where $\Lambda$ is a positive constant independent of $x$. 

For a sufficiently small time $t$, the point $\gamma(t)$ belongs to $\Omega_x$ and we can associate to $\gamma(t)$ coordinates in the cone $C_{[0,kt]}(S_x)$, which we denote $(r(t), \theta(t))$, with $\theta(t)$ in $S_x$. We aim to show that these coordinates in the tangent cone admit a unique limit $\dot{\gamma}(0)=(0,\theta(0))$ as $t$ tends to zero. 

For what concerns the radial coordinate $r$ the situation is simpler. Thanks to the inequality $\eqref{distanceSommet}$ we have:
\begin{equation*}
|\dist(x, \gamma(t)) - \mbox{d}_{C}((0, \theta(0)), (r(t), \theta(t))| \leq \Lambda t^{1+\alpha}.
\end{equation*}
Since $\gamma$ is a minimizing geodesic and by using the expression $\eqref{ConeDist}$ for the distance in the cone, we get:
\begin{equation*}
|t-r(t)| \leq \Lambda t^{1+\alpha},
\end{equation*}
which means that the radial coordinate satisfies $r(t)=t+ O(t^{1+\alpha})$. As a consequence, $r(t)$ easily converges to zero as $t$ goes to zero. For simplicity, from now on in the proof we will replace $r(t)$ by $t$: we leave to the reader the straightforward computation with $t+O(t^{1+\alpha})$.

It remains to show that $\theta(t)$ converges to a unique point $\theta(0)$ in $S_x$. Since $S_x$ is compact, we know that for any sequence $t_j$ going to zero, there exists a subsequence such that $\theta(t_j)$ converges to a point in $S_x$. We want to prove that for any two sequences $t_j$, $s_j$ tending to zero, such point is the same. 

Consider $t, s>0$ sufficiently small. Then $\gamma(t)$ and $\gamma(s)$ belongs to a ball centred at $x$ of radius equal to the maximum between $t$ and $s$. As we recalled above, such ball is included in an open neighbourhood of $x$ homeomorphic to a truncated cone over $S_x$. The estimate for the metrics together with the fact that $\gamma$ is minimizing lead to the following: 
\begin{equation}
\label{diffDist0}
|\dist(\gamma(t), \gamma(s))-\sqrt{t^2+s^2-2st\cos \distArg{S_x}(\theta(t), \theta(s))}|\leq \Lambda\max\left\{t,s\right\}^{\alpha}\dist(\gamma(t), \gamma(s))
\end{equation}
which can be rewritten as:
\begin{equation}
\label{diffDist}
\left|1-\sqrt{1+4\frac{st}{|t-s|^2}\sin^2 \left( \frac{\distArg{S_x}(\theta(t), \theta(s))}{2}\right)}\right|\leq \Lambda \max\{t,s\}^{\alpha}. 
\end{equation}
Consider the sequence $t_j=2^{-j}$: first, we are going to show that the sequence $\theta(t_j)$ converges to a point $z_0$ in $S_x$ without passing to a subsequence. This is done by proving that $\theta(t_j)$ is a Cauchy sequence. Then, we are going to prove that for any other sequence $s_j$ tending to zero as $j$ goes to infinity, $\theta(s_j)$ converges to $z_0$ as well. 

In the inequality $\eqref{diffDist}$ replace $t=t_j$ and $s=t_{j+1}$. We then obtain:
\begin{equation*}
\left| 1-\sqrt{1+8\sin^2\left( \frac{\distArg{S_x}(\theta(t_j), \theta(t_{j+1}))}{2}\right)}\right|\leq 2\Lambda \left(\frac{1}{2^{\alpha}}\right)^j
\end{equation*} 
This implies that the distance between $\theta(t_j)$ and $\theta(t_{j+1})$ converges to zero as $j$ tends to infinity. More precisely, by multiplying by the conjugate quantity and by using the Taylor expansion of sine at zero, we can state that there exists a positive constant $C$ such that: 
\begin{equation*}
\distArg{S_x}(\theta(t_j), \theta(t_{j+1})) \leq C \left(\frac{1}{2^{\alpha}}\right)^j. 
\end{equation*}
The sequence $2^{-\alpha j}$ is such that its series converges and therefore $\theta(t_j)$ is a Cauchy sequence. Then it converges to a point $z_0$ in $S_x$, without passing to any subsequence. 
Now consider a sequence $s_i$ going to zero as $i$ tends to infinity. We need to prove that $\theta(s_i)$ converges to $z_0$. For any $i$, choose $j_i$ in $\N$ such that $2^{-j_i-1} \leq s_i < 2^{-j_i}$. Then by the triangular inequality we have:
\begin{equation*}
\distArg{S_x}(\theta(s_i), z_0) \leq \distArg{S_x}(\theta(s_i), \theta(t_{j_i}))+ \distArg{S_x}(\theta(t_{j_i}), z_0). 
\end{equation*}
We know that the second term in the right-hand side tends to zero as $j_i$ goes to infinity. As for the first term consider the inequality $\eqref{diffDist}$ with $t=t_{j_i}$ and $s=s_i$. We multiply and divide the left-hand side of $\eqref{diffDist}$ by the conjugate quantity: 
\begin{equation}
\label{ConjTheta}
\frac{\frac{4ts}{|t-s|^2}\sin^2\left( \frac{\distArg{S_x}(\theta(t), \theta(s))}{2}\right)}{1+\sqrt{1+\frac{4ts}{|t-s|^2}\sin^2\left( \frac{\distArg{S_x}(\theta(t), \theta(s))}{2}\right)}}\leq \max\{s,t\}^{\alpha}=2^{-\alpha j_i}.
\end{equation}
Denote by $\rho$ the numerator of this expression and rewrite the previous as:
\begin{equation*}
f(\rho)=\frac{\rho}{1+\sqrt{1+\rho}} \leq 2^{-\alpha j_i}.
\end{equation*}
For $j_i$ sufficiently large, the right-hand side of this inequality is smaller than one. Since the function $f$ is increasing and $f(3)=1$, we get that $\rho$ belongs to the interval $(0,3)$. Then again by using the previous inequality we obtain:
\begin{equation*}
\rho \leq 2^{-\alpha j_i}(1+\sqrt{1+\rho}) \leq 3 \cdot 2^{-\alpha j_i}. 
\end{equation*}
Getting back to $\eqref{ConjTheta}$, we have obtained: 
\begin{equation*}
\sin^2\left(\frac{\distArg{S_x}(\theta(t), \theta(s))}{2}\right) < 3 \cdot 2^{-\alpha j_i} \frac{|t-s|^2}{4ts}. 
\end{equation*}
Now, thanks to our choice of $t$ and $s$ we have the following bounds:
\begin{equation*}
2^{-2j_i-1} \leq  ts < 2^{-2j_i},  \quad \quad 
|t-s| <2^{-j_i-1},
\end{equation*}
which imply that for some positive constant $C_1$ we have:
\begin{equation*}
\sin^2\left(\frac{\distArg{S_x}(\theta(t), \theta(s))}{2}\right) \leq C_1 2^{-\alpha j_i}
\end{equation*}
We have then shown that the distance in $S_x$ between $\theta(s_i)$ and $\theta(t_{j_i})$ must tend to zero as $i$ tend to infinity. Therefore $\theta(s_i)$ converges to $z_0$, and this is true for any sequence $\{s_i\}$ tending to zero. This proves that $\theta(0)$ in $S_x$, and then $\dot{\gamma}(0)$ in $C(S_x)$, are well defined and unique, as we wished.  
\end{proof}

\begin{lemma}
\label{Lemma2}
Let $(X,g)$ be an admissible stratified space, $\gamma: [-\eps, \eps] \rightarrow X$ a minimizing geodesic and let $x$ be the point $\gamma(0)$. Then the diameter of the tangent sphere $S_x$ is equal to $\pi$. 
\end{lemma}

\begin{proof}
As we recalled above, for each point $x$ of $X$ the tangent sphere $S_x$ is an admissible stratified space, and then by the singular Myers theorem we know that its diameter is less or equal than $\pi$. As a consequence, it suffices to find two points in $S_x$ at distance $\pi$. As we did in the previous proof, for a time $t$ small enough, we can associate to $\gamma(t)$ the coordinates $(r(t), \theta(t))$ in $C(S_x)$. Observe that $r(t)$ belongs to $\R_+$, and since we are considering negative values for $t$, if we repeat the same argument as above for the variable $r(t)$ we get $r(t)=|t|+O(t^{1+\alpha})$. 

We claim that the two points at distance $\pi$ in $S_x$ are given by:
\begin{equation*}
\theta_+=\lim_{t \rightarrow 0^+} \theta(t), \quad \quad \theta_-=\lim_{t \rightarrow 0^-} \theta(t)
\end{equation*}
Both of the two limits exist in $S_x$ thanks to the previous Lemma. 

Fix $t>0$, consider $\theta(t)$ and $\theta(-t)$. By using $\eqref{diffDist0}$ and again for simplicity by replacing $r(t)$ by $|t|$, we have the following:
\begin{equation*}
\left| 2t - \sqrt{2t^2-2t^2\cos \distArg{S_x}(\theta(t), \theta(t))} \right| \leq 2\Lambda t^{1+\alpha}. 
\end{equation*}
Then we can divide both sides of the inequality by $2t$ and get:
\begin{equation*}
\left| 1 -\sin\left(\frac{\distArg{S_x}(\theta(t), \theta(-t))}{2} \right) \right|= \left| 1-\sqrt{\frac{1-\cos(\distArg{S_x}(\theta(t), \theta(-t))}{2}} \right| \leq \Lambda t^{\alpha}. 
\end{equation*}
As a consequence, when $t$ tends to zero, the distance in $S_x$ between $\theta(t)$ and $\theta(-t)$ must tend to $\pi$, and we get $\distArg{\mbox{\tiny{$S_x$}}}(\theta_+,\theta_-)=\pi$. Then the tangent sphere has diameter equal to $\pi$. 
\end{proof}

\begin{lemma}
\label{Lemma3}
Let $(X^n,g)$ be an admissible stratified space of diameter equal to $\pi$. Let $P$ a point in $X$ such that there exists $N$ in $X$ at distance $\pi$ from $P$. For any point $x_0$, distinct from $P$, if $\gamma_1, \gamma_2$ are respectively minimizing geodesics from $P$ to $x_0$ and from $x_0$ to $N$, then the concatenation of $\gamma_1$ and $\gamma_2$ is a minimizing geodesic from $P$ to $N$. 
\end{lemma}

\begin{proof}
Thanks to the Myers singular theorem \ref{MyersSing}, we know that the first non-zero eigenvalue of the Laplacian is equal to the dimension $n$, and moreover that the function: 
\begin{equation*}
\varphi(x)=\sin{\left(\mbox{d}_g(x,P)-\frac{\pi}{2}\right)}=\cos(\mbox{d}_g(x,P)): X \rightarrow [-1,1]
\end{equation*}
is an eigenfunction for the Laplacian associated to $n$. Let $P, N, x_0$ and $\gamma_1$, $\gamma_2$ be as in the statement. To show that the concatenation $\gamma$ of $\gamma_1$ and $\gamma_2$ is a minimizing geodesic from $P$ to $N$, it suffices to prove that the sum of $\mbox{d}_g(x_0,P)$ and $\dist(x_0,N)$ is equal to $\pi$. Let us consider 
\begin{equation*}
\varphi_N(x)=\cos(\dist(x,N)),
\end{equation*}
which is again an eigenfunction associated to $n$. Assume that the distance from $x_0$ to $P$ is less than $\pi/2$. Denote:
\begin{equation*}
\mathcal{U}_+=\left\{x \in X \mbox{ s.t. } d(x,P) < \frac{\pi}{2} \right\}= \left\{ x \in X \mbox{ s.t. } \varphi(x) > 0\right\}.
\end{equation*}
Then the distance between all points in $\mathcal{U}_+$ and $N$ is larger than $\pi/2$, and $\varphi_N$ is negative on $\mathcal{U}_+$. We are going to use the same integration by parts as we did in the proof of Thoerem \ref{MyersSing}. For any $\eps>0$ define 
\begin{equation*}
v_{\eps}=\frac{\varphi_P}{\eps - \varphi_N},
\end{equation*}
which is a positive function on $\mathcal{U}_+$ and belongs to $W^{1,2}_0(\mathcal{U}_+)$. Consider $v_{\eps}\varphi_N$ and the norm in $L^2$ of its gradient:
\begin{align}
\label{IPPàgogo}
\int_{\mathcal{U}_+} |d(v_{\eps}\varphi_N)|^2 dv_g &= \int_{\mathcal{U}_+} (|dv_{\eps}|^2 \varphi_N^2+2\varphi_N v_{\eps}(dv_{\eps}, d\varphi_N)_g+ v_{\eps}^2|d\varphi_N|^2) dv_g \\
& \geq \int_{\mathcal{U}_+} (d(v_{\eps}^2\varphi_N), d\varphi_N)_gdv_g = \int_{\mathcal{U}_+} v_{\eps}^2\varphi_N \Delta_g \varphi_N dv_g. 
\end{align}
Now, $\varphi_N$ is an eigenfunction of the Laplacian associated to the eigenvalue $n$, and then we obtain:
\begin{equation*}
\int_{\mathcal{U}_+} |d(v_{\eps}\varphi_N)|^2 dv_g \geq n \int_{\mathcal{U}_+} v_{\eps}^2 \varphi_N^2 dv_g.
\end{equation*}
When we let $\eps$ tend to zero, by the dominated convergence theorem, we get:
\begin{equation}
\int_{\mathcal{U}_+} |d(\varphi_P)|^2 dv_g \geq n \int_{\mathcal{U}_+} \varphi_P^2 dv_g. 
\end{equation}
But thanks to Theorem \ref{MyersSing} we already know that the equality is attained for $\varphi_P$, and therefore we have equality in each line of $\eqref{IPPàgogo}$. This implies that $dv_0$ vanishes and $v_0$ is constant on each connected component of $\mathcal{U}_+$, and since $\mathcal{U}_+$ is connected, the quotient $v_0=-\varphi_P/ \varphi_{N}$ is constant on $\mathcal{U}_+$. Both $-\varphi_N$ and $\varphi_P$ take values between $0$ and $1$ on $\mathcal{U}_+$, and as a consequence the constant must be equal to one. 
We have shown that for any $x$ in $\mathcal{U}_+$ we have:
\begin{equation*}
\varphi(x)=\cos(\dist(x,P))=-\cos(\dist(x,N))=-\varphi_N(x). 
\end{equation*}
Which implies that, in particular, $\dist(x_0,P)+\dist(x_0,N)=\pi$. If the distance between $x_0$ and $P$ is larger than $\pi/2$ we can repeat the same argument by exchanging the roles of $P$ and $N$. It remains to study the case in which $x_0$ is at distance equal to $\pi/2$ from $P$. Observe that for any $x$ in $X$ we have:
\begin{equation*}
\dist(x,P)+d(x,N) \geq \pi,
\end{equation*}
and since the cosine is a decreasing function on $[0, \pi]$ we get:
\begin{equation*}
\varphi_N(x)=\cos(\dist(x,N))\leq \cos(\pi-\dist(x,P))=-\cos(\dist(x,P))=-\varphi_P(x). 
\end{equation*}
We have proven in particular that the equality holds in the sets in which $\varphi_N$, $\varphi_P$ do not vanish. If $x_0$ is such that $\varphi_P(x_0)=0$, assume by contradiction that $\varphi_N(x_0)>0$. Thus $x_0$ belongs to the set in which $\varphi_N$ is strictly positive, and we have shown that in this set $\varphi_N$ coincides with $-\varphi_P$. This would imply that $\varphi(x_0)$ is strictly negative, which it is not, and therefore we have proven that $\varphi_P$ and $\varphi_N$ vanish in the same set. This means that if $x_0$ is a distance $\pi/2$ from $P$, then $\dist(x_0,N)$ is equal to $\pi/2$ as well. This concludes the proof.
\end{proof} 


We are now in position to prove Theorem \ref{ObataSing}. 

\begin{proof}[Proof of Theorem \ref{ObataSing}]
One of the two implications is trivial. In fact, if we consider an admissible stratified space $\hat{X}$ of dimension $(n-1)$ and its spherical suspension, the function $\varphi(t)=\sin(t)$ is an eigenfunction with associated eigenvalue $n$. 

Our proof of the other implication is by induction on the dimension of $X$. If $n$ is equal to 1, $X$ is a circle with metric $a^2d\theta^2$ for $a\leq 1$, and then the first eigenvalue of the Laplacian is equal to one if and only if $a$ is equal to one. Assume that we have proven the statement of the theorem for all dimensions $k$ until $(n-1)$ and let $X^n$ be an admissible stratified space of dimension $n$ with diameter equal to $\pi$. The induction hypothesis, together with the previous lemmas, leads to an important consequence on the tangent cones. Let $P$ and $N$ be two antipodal points. Thanks to Lemma \ref{Lemma3}, we know that any point $x$ in $X$, distinct from $P$ and $N$, belongs to the interior of a minimizing geodesic from $P$ to $N$. Then Lemma \ref{Lemma2} implies that the tangent sphere $S_x$ at $x$ has diameter equal to $\pi$. Therefore by the induction hypothesis $S_x$ is isometric to the spherical suspension of an admissible stratified space $(Y,k)$ of dimension $(n-2)$: we can apply Corollary \ref{splitting} to the tangent cone $C(S_x)$ in order to deduce that $C(S_x)$ is isometric to the product $\R\times C(Y)$. If $Y$ has diameter equal to $\pi$, we can iterate this argument and, as we observed in Remark \ref{m_split} we get that $C(S_x)$ is isometric to $\R^m \times C(Y_0)$, where $m \geq 1$ and $(Y_0, k_0)$ is an admissible stratified space of dimension $(n-m-1)$ with diameter strictly smaller than $\pi$. Observe that, since there is no singular stratum of codimension 1, $m$ is either between $1$ and $(n-2)$, and $x$ belongs to the singular set $\Sigma$, or $m=(n-1)$ and $C(Y_0)$ is the real line $\R$, and $x$ is a regular point.

Let us denote $f(x)=\dist(x,P)-\pi/2$. We consider the set of regular points that are equidistant from $P$ and $N$:
\begin{equation*}
\Gamma_0= \left\{ x\in \Omega : d(x,N)=d(x,P) \right\}.
\end{equation*} 
Observe that $\Gamma_0$ also coincides with the subset of the regular set in which $\varphi$ and $f$ vanish, and thus it is not empty. 

Our first goal is to show that any point in $\Gamma_0$ possesses a neighbourhood which is isometric to the product of a neighbourhood $\mathcal{V}$ in $\Gamma_0$ with some small interval, endowed with the appropriate warped product metric. This will show that the metric $g$ locally has the desired form. Then we aim to prove that the regular set $\Omega$ is isometric to $\Gamma_0 \times \left[ -\pi/2, \pi/2\right]$ endowed with a warped product metric. Finally, we will extend the isometry to the whole of $X$ and show that the closure of $\Gamma_0$ with respect to the metric $g$ is in fact a stratified space. \newline

\emph{Step 1}. Let us denote $\hat{g}$ the metric $g$ restricted to $\Gamma_0$. We show that for any $x$ in $\Gamma_0$ there exist a closed neighbourhood $\mathcal{W}$ of $x$ in $X$, a closed neighbourhood $\mathcal{V}$ of $x$ in $\Gamma_0$ and an interval $[0,T_x)$ such that the metric $g$ on $\mathcal{W}$ is isometric to the warped product metric $dt^2+\cos^2(t)\hat{g}$ on $\mathcal{V}\times [0,T_x)$. The argument that we use is similar to the one developed in Proposition 5.1 of \cite{BourCarron} in order to study the case of equality in the refined Kato inequality for 1-forms.

Observe that on the regular set $\Omega$ the gradient $\nabla f(x)$ is well-defined, it has norm equal to one and is the unit normal vector field of the level hypersurface $f^{-1}(f(x))\cap \Omega$. Then for each point $x \in \Gamma_0$ there is a compact neighbourhood $\mathcal{V}$ of $x$, closed in $\Gamma_0$, and an interval $[0,T_x)$ on which the flow $\gamma_x(t)$ of the gradient exists. Since $\mathcal{V}$ includes a closed ball centred in $x$ of radius sufficiently small, we can restrict our study to such ball, and from now on $\mathcal{V}$ is a closed ball in $\Gamma_0$ centred at $x$. Observe that $\gamma_x$ is a minimizing geodesic on $ [0,T_x]$.

The time $T_x$ is defined as follows. For each $y$ in $\mathcal{V}$ we can consider the minimal time of existence for the flow $\gamma_y$, that is:
\begin{equation*}
T(y)=\inf\left\{t>0 \mbox{ such that } \gamma_y(t) \mbox{ belongs to }\Sigma \right\}
\end{equation*}
Then $T_x$ will be the infimum of all these times over $\mathcal{V}$:
\begin{equation*}
T_x=\inf_{y \in \mathcal{V}}T(y). 
\end{equation*}
This means that $T_x$ is the smallest time for which the flow of $\nabla f$ starting at a point of $\mathcal{V}$ intersects the singular set. The function $T(y)$ is lower semi-continuous, and therefore it has a minimum on the compact neighbourhood $\mathcal{V}$: this means that there exists $y_0$ in $\mathcal{V}$ such that $T(y_0)=T_x$. Let us denote $x_0$ the point in $\Sigma$ such that $\gamma_{y_0}(T_x)=x_0$.

By a classical result contained in \cite{Milnor} we get the diffeomorphism:
\begin{align*}
E: \mathcal{V} \times [0,T_x) &\rightarrow f^{-1}([0,T_x)) \cap \Omega \\
E(x,t) &=\gamma_x(t). 
\end{align*}
Then we obtain an isometry if we equip $\mathcal{V} \times [0,T_x)$ with the pull-back metric $E^*g$. We can easily extend this isometry to $\mathcal{V}\times \{T_x\}$. In fact, for any $y$ in $\mathcal{V}$ we can define:
\begin{equation*}
E(y,T_x)=\lim_{t \rightarrow T_x} E(y,t). 
\end{equation*}
This limit exists since $X$ is compact, thus complete, and the function $t \mapsto E(x,t)$ is Lipschitz with Lipschitz constant equal to one. Moreover, since $f$ is continuous, we know that for any $x$ in $\mathcal{V}$ the point $E(y,T_x)$ belongs to $f^{-1}(T_x)$. Then we have obtained an isometry $E$ between the product $\mathcal{V}\times [0,T_x]$ endowed with the metric $E^*g$ and a closed neighbourhood $\mathcal{W}$ of $x$ which is included in $f^{-1}([0,T_x])$:
\begin{equation*}
E: (\mathcal{V} \times [0,T], E^{*}g) \rightarrow (\mathcal{W},g). 
\end{equation*}
We claim that the level hypersurfaces $\mathcal{V} \times \left\{t\right\}$ are umbilical for any $t \in [0,T_x)$. In order to show this observe that $E$ sends $\mathcal{V} \times \left\{t\right\}$ to a regular subset of the inverse image of $t$ via $f$, which we denote $\Gamma_t=f^{-1}(t)\cap \Omega$. Recall that, by definition of $f$, $\Gamma_t$ is the set of regular points which are at distance equal to $(t+ \pi/2)$ from $P$. As a consequence we have that the function $\varphi \circ E$ only depends on $t$:
\begin{equation*}
\varphi(E(x,t))=\cos\left(\dist(E(x,t),P)+\frac{\pi}{2}\right)=\sin(t). 
\end{equation*}
Moreover, $\varphi$ is an eigenfunction relative to the eigenvalue $n$, and thus its Hessian must satisfy the equality $\nabla d\varphi=-\varphi g$: if we look at this relation in the coordinates given by the isometry $E$ we get:
\begin{equation*}
E^*(\nabla d\varphi)=-\sin(t)dt\otimes dt +\cos(t)\nabla dt=-\sin(t) E^*g. 
\end{equation*}
As a consequence we obtain:
\begin{equation*}
\nabla dt=-\tan(t) E^*g.
\end{equation*}
This shows that the Hessian of the hypersurfaces $\mathcal{V}\times \{t\}$ is proportional to the metric, therefore that $\mathcal{V}\times \{t\}$ is umbilical for any $t \in [0, T_x)$. As a consequence, there exists a function $\eta$ such that the metric $E^*g$ on $\mathcal{V}\times [0,T_x)$ is equal to $dt^2+\eta(t)^2\hat{g}$, where $\hat{g}$ is the metric $g$ restricted to $\Gamma_0$. But thanks to the previous equality on the Hessian we know that $\eta$ must satisfy:
\begin{equation*}
\frac{\eta'(t)}{\eta(t)}=-\tan(t), \quad \eta(0)=1. 
\end{equation*}
Therefore we deduce that $\eta(t)=\cos(t)$. We have then proven that, locally, the metric $g$ is isometric to the warped product metric:
\begin{equation*}
E^*g =dt^2+\cos^2(t)\hat{g}. 
\end{equation*}
\newline
\emph{Step 2}. We aim to show that for any $x$ in $\Gamma_0$ the time $T_x$ must be equal to $\pi/2$, or, in other words, that for any $y \in \mathcal{V}$ the geodesic $\gamma_y(t)$ cannot intersect the singular set before getting to a point at distance $\pi$ from $P$. This will allow us to extend the isometry $E$ to the product of $\Gamma_0$ and the interval $[-\frac{\pi}{2}, \frac{\pi}{2}]$. We assume by contradiction that $T_x$ is strictly smaller than $\pi/2$, and we prove that as a consequence $x_0$ must belong to the regular set. In order to do that, we are going to compare the spherical geometry of $\mathcal{V}\times [0,T_x)$ with the geometry of the tangent cone at $x_0$. 

Observe that, if we consider a minimizing geodesic from $P$ to $y_0$ and its concatenation $\gamma$ with $\gamma_{y_0}$, this gives a minimizing geodesic from $P$ to $x_0$, because $x_0$ is exactly at distance $T_x+\pi/2$ from $P$. Lemma \ref{Lemma3} ensures that $\gamma$ can be continued to a minimizing geodesic from $P$ to $N$. Moreover, as we stated above, Lemma \ref{Lemma2} and the induction assumption imply that the tangent cone at $x_0$ is isometric to $\R \times C(Y)$, where the first coordinate in this decomposition is the Busemann function associated to a geodesic joining the vertex of the cone with two antipodal points in $S_{x_0}$. If the diameter of $Y$ is equal to $\pi$, $C(S_{x_0})$ is isometric to $\R^m \times C(Y_0)$, where $Y_0$ is an admissible stratified space with diameter strictly less than $\pi$ and $m$ is between $1$ and $(n-2)$.

The point $y_0$ can belong either to the interior or to the boundary of $\mathcal{V}$. Let us assume that $y_0$ belongs to the boundary of $\mathcal{V}$: the other case will follow easily. Let $\eps$ and $\delta$ be two positive real numbers, sufficiently small, with $\delta << \eps$. Let us consider $x_{\delta}=\gamma_{y_0}(T-\delta)$. If we consider a ball $B(x_0, \eps)$ centred at $x_0$ of radius $\eps$, we know that the truncated tangent cone at $x_0$ is the following pointed Gromov-Hausdorff limit as $\eps$ goes to zero:
\begin{equation*}
C_{[0,1)}(S_{x_0})=\lim_{\eps \rightarrow 0}(B(x_0,\eps), \eps^{-2}g, x_0). 
\end{equation*}
Moreover, the ball $B(x_0,\eps)$ can be seen as the Gromov-Hausdorff limit of the ball $B(x_{\delta}, \eps)$ as $\delta$ goes to zero. In fact, the Gromov-Hausdorff distance between the two balls is less than or equal to the distance between $x_{\delta}$ and $x_0$, which eventually tends to zero. We can write:
\begin{equation*}
C_{[0,1)}(S_{x_0})= \lim_{\eps \rightarrow 0}\lim_{\delta \rightarrow 0}(B(x_{\delta},\eps), \eps^{-2}g, x_0)
\end{equation*} 
Since $x_{\delta}$ belongs to the regular set and we have the isometry $E$, we know part of the geometry of the ball $B(x_{\delta}, \eps)$. More precisely, for $\delta << \eps$ consider a ball $B(y_0, \eps -\delta)$ in $\Gamma_0$ and denote by $B^+(y_0,\eps-\delta)$ the part of this ball which intersects $\mathcal{V}$: if $\eps$ is small enough we can parametrize $B^+(y_0, \eps-\delta)$ by 
\begin{equation*}
([0,\eps-\delta)\times \s^{n-2}_+, \hat{g}=d\rho^2+\rho^2d\sigma_{n-2}^+ + o(\rho^2)),
\end{equation*}
where $\s^{n-2}_+$ is the upper half sphere of dimension $(n-2)$. 
The image via $E$ of the product \mbox{$B^+(y_0, \eps-\delta) \times (T_x-\eps -\delta, T_x-\delta]$} is contained in $B(x_{\delta}, \eps)$, and it is endowed of the metric: 
\begin{equation*}
g_{\delta}=ds^2+\cos^2(T-\delta+s)\hat{g}. 
\end{equation*}
Observe that in case $y_0$ belongs to the interior of $\mathcal{V}$ one can just consider the whole ball of radius $(\eps-\delta)$ around $y_0$, which is included in $\mathcal{V}$ for $\eps$ and $\delta$ small enough. 

Our goal is to study the limit as $\delta$ goes to zero of the product between $B^+(y_0, \eps-\delta)$ and the interval $(T_x-\eps -\delta, T_x-\delta]$ endowed with the metric $g_{\delta})$. Then we rescale the metric by a factor $\eps^{-2}$ and pass to the limit as $\eps$ goes to zero. This will give a subset of the tangent cone at $x_0$ and will allow to deduce further information on its geometry. If we consider the interval $(T_x-\eps-\delta, T_x-\delta]$ is because the isometry $E$ is defined until $T_x$, and therefore we have information on the metric only in the regular part of $\mathcal{W}$, which precedes the point $x_0$. 

As $\delta$ goes to zero, the metric $g_{\delta}$ on $[0,\eps-\delta)\times \mathbb{B}^+_{y_0}$ converges in $C^{\infty}$ to the metric $ds^2+\cos^2(T+s)\hat{g}$ on $[0,\eps)\times \mathbb{B}^{+}_{y_0}$. This limit is in particular a Gromov-Hausdorff limit. If we consider the changes of coordinates $s=\eps r$ and $\rho = \eps \tau$ for $r, \tau \in [0,1)$, it is easy to see that $[0,\eps)\times \mathbb{B}^+_{y_0}$ endowed with the rescaled metric $\eps^{-2}g$ converges in the Gromov-Hausdorff sense to:
\begin{equation*}
H=[0,1) \times [0,1) \times \s^{n-2}_+
\end{equation*}
endowed with the metric:
\begin{equation*}
dr^2+d\tau^2+\tau^2d\sigma_{n-2}^+.
\end{equation*}
As a consequence the tangent cone $C_{[0, 1)}(S_{x_0})$ includes a subset isometric to $H$. Since the convergence is a pointed Gromov-Hausdorff convergence and preserve the base point $x_0$, a subset $H_0$ isometric  to the product $\R_+\times \R_+ \times \s^{n-2}_+$ is included in $C(S_{x_0})$.

Recall that in $X$ the variable $r$ was chosen to be equal to $s/\eps$, where $s$ is the distance between a point and $x_0$ along the geodesic $\gamma_{y_0}$. There exists a limit for $\gamma_{y_0}$ in the tangent cone which is a minimizing geodesic $\gamma_0$ in $C(S_{x_0})$ starting from the vertex $x_0$. Since $\gamma_{y_0}$ can be continued until $N$, the minimizing limit geodesic $\gamma_0$ is defined on the whole $\R$ and it connects the vertex $x_0$ with two antipodal points in $S_{x_0}$. As a consequence, in the splitting $\R\times C(Y)$ of the tangent cone $C(S_{x_0})$, the first coordinate is the opposite of the Busemann function $B_{\gamma_0}$ associated to $\gamma_0$. Now, when we look at $r$ in this limit of $s/\eps$ as $\eps$ tends to zero, it is possible to show that $r$ on $H_0$ coincides with $-B_{\gamma_0}$, that is:
\begin{equation*}
r(x)=\lim_{t \rightarrow +\infty} (t-\mbox{d}_{C(S_{x_0})}(x, \gamma_0(t))).  
\end{equation*} 
Indeed we have the following:
\begin{align*}
r(\cdot) & = \lim_{\eps \rightarrow 0} \frac{s(\cdot)}{\eps} = \lim_{\eps \rightarrow 0} \left( - \frac{T_x-s(\cdot)-T_x}{\eps}\right)=\\
&=\lim_{\eps \rightarrow 0} \left( - \frac{\dist(y_0, \gamma_{y_0}(T_x-s(\cdot))-\dist(y_0, x_0)}{\eps} \right). 
\end{align*}
Now observe that $\eps^{-1}\dist(y_0, x_0)$ tends to infinity, and to the distance in the tangent cone $C(S_{x_0})$ from the vertex $x_0$. The geodesic $\gamma_{y_0}$ converges to the limit geodesic $\gamma_0$, and therefore we get: 
\begin{equation*}
r(\cdot)= \lim_{t \rightarrow +\infty} -\left( \mbox{d}_{C(S_{x_0})}(\cdot, \gamma_0(t))- \mbox{d}_{C(S_{x_0})}(\cdot, x_0) \right)= \lim_{t \rightarrow +\infty} -\left( \mbox{d}_{C(S_{x_0})}(\cdot, \gamma_0(t))-t \right). 
\end{equation*}
We have shown that $r$ coincides with minus the Busemann function of the minimizing geodesics $\gamma_0$, and therefore we can extend it from $H_0 \cong \R_+\times\R_+\times \s^{n-2}_+$ to the whole tangent cone $C(S_{x_0})$. Moreover, recall that $B_{\gamma_{0}}$ is onto on $\R$, and so it is the extension of $r$. Therefore, the tangent cone $C(S_{x_0})$ includes a subset which is isometric to $\R\times \R_+\times \s^{n-2}_+$, which is isometric in turns to  $\R_+ \times \R^{n-1}$ endowed with the product metric. We also know that $C(S_{x_0})$ is isometric to $\R^{m}\times C(Y_0)$. Then the previous discussion shows that $m$ must be equal to $(n-1)$ and $C(Y_0)$ is a stratified space of dimension 1 without boundary. The only possible choice for $C(Y_0)$ is that it is a line $\R$. Therefore we have proven that the tangent cone at $x_0$ is isometric to  $\R^n$ and that $x_0$ must belong to the regular set of $X$.


As a consequence, for any point in $\mathcal{V}$ the flow of $\nabla f$ is defined on the interval $[0, \pi/2)$, and since the above discussion is independent of the choice of $x$ in $\Gamma_0$ we can define the isometry $E$ on the product $\Gamma_0 \times [0, \pi/2)$:
\begin{align*}
E: \left(\Gamma_0 \times \left[0, \frac{\pi}{2} \right), E^*g\right) & \rightarrow \left(f^{-1}\left(\left[ 0, \frac{\pi}{2} \right) \right) \cap \Omega, g \right), \\
E(x,t) &= \gamma_x(t). 
\end{align*} 
We can also extend $E$ to the closed interval, as we did above, by defining:
\begin{equation*}
\hat{E}\left(x,\frac{\pi}{2}\right)=\lim_{t\rightarrow \frac{\pi}{2}}E(x,t)
\end{equation*}
Observe that for each $x$ in $\Gamma_0$ the endpoint of $\gamma_x$ is a point at distance $\pi$ from $P$, but it is not necessarily the same point for all $x$ in $\Gamma_0$. But thanks to the fact that $E^*g$ is a warped product metric, the image of $\Gamma_0\times \{\pi/2\}$ via $\hat{E}$ consists of only one point. In fact, consider a curve $\gamma$ in $\Gamma_0$ of length $L$ with respect to $g$. For any $t \in [0,\pi/2)$ the length of $\gamma \times \{t\}$ in $\Gamma_0 \times [0, \pi/2]$ endowed with the metric $E^*g$ is equal to $\cos^2(t)L$, and since $E$ is an isometry we have:
\begin{equation*}
L_g(\hat{E}(\gamma,t)) = \cos^2(t) L_g(\gamma) \leq L_g(\gamma). 
\end{equation*}
As a consequence, when $t$ is equal to $\pi/2$, the length of the image via $\hat{E}$ of $\gamma \times \{\pi/2\}$ is equal to zero, which means that $\hat{E}(\Gamma_0, \pi/2)$ has diameter equal to zero, and therefore it consists of only one point at distance $\pi$ from $P$. We denote again this point as $N$. 

We have obtained an isometry $\hat{E}$:
\begin{equation*}
\hat{E}: \left(\Gamma_0 \times \left[0, \frac{\pi}{2}\right], E^*g \right) \rightarrow \left( \left(f^{-1}\left(\left[ 0, \frac{\pi}{2} \right] \right) \cap \Omega \right) \cup \{ N \}, g \right). 
\end{equation*}

The same argument can be repeated for negative values of $t$, in order to show that for any $x$ in $\Gamma_0$ the geodesic flow of $\nabla f$ exists for $t\in \left( -\frac{\pi}{2},0\right]$ and does not intersect the singular set between $x$ and $P$. Then we have an isometry $\hat{E}$:
\begin{equation*}
\hat{E}: \left( \Gamma_0 \times \left[-\frac{\pi}{2},\frac{\pi}{2} \right], E^*g=dt^2+\cos^2(t)\hat{g}\right) \rightarrow (\Omega \cup \{P,N\},g). 
\end{equation*}
\newline
\emph{Step 3}. We finally prove that the metric completion of $\Gamma_0$ with respect to the metric \mbox{$E^*g= dt^2+\cos^2(t)\hat{g}$} is a stratified space. This is done by studying the geometry of the tangent cone at $P$. Consider $\eps > 0$ and a neighbourhood $B(P, \eps)$ of $P$. The isometry $\hat{E}$ restricts to an isometry between:
\begin{equation*}
\left[-\frac{\pi}{2}, -\frac{\pi}{2}+\eps\right) \times \Gamma_0 \rightarrow (B(P,\eps) \cap \Omega) \cup \{P\} = B(p,\eps)^{\mbox{\tiny{reg}}} \cup\{P\}.
\end{equation*}
If we consider the pointed Gromov-Hausdorff limit of $(B(P,\eps)^{\mbox{\tiny{reg}}} \cup\{P\}, P, \eps^{-2}g)$ for $\eps$ going to zero, the definition of the tangent cone and the fact that the convergence of the metrics is uniform in the Lipschitz topology ensure that we obtain the cone $(C(S_P^{\mbox{\tiny{reg}}}), ds^2+s^2h_P)$, where $(S_P, h_P)$ is the tangent sphere at $P$.

We can consider as well the limit for $\eps$ going to zero of:
\begin{equation}
\label{truc}
\left( \left[-\frac{\pi}{2}, -\frac{\pi}{2}+\eps\right) \times \Gamma_0, \hat{E}^{-1}(P), \eps^{-2}(dt^2+\cos^2(t) \hat{g}) \right)
\end{equation}
We change the variable $t=-\pi/2 +s\eps$ and by taking the Taylor expansion of sine in $0$ we obtain: 
\begin{equation*}
\eps^{-2}(\eps^2ds^2+\sin^2(\eps s)\hat{g}) \rightarrow ds^2 +s^2 \hat{g} \quad \mbox{ as } \eps \rightarrow 0.
\end{equation*}
Therefore the pointed Gromov-Hausdorff limit of $\eqref{truc}$ as $\eps$ goes to zero is the cone \mbox{$( C(\Gamma_0), 0, ds^2+s^2\hat{g})$}, where we denoted with $0$ its vertex. Moreover, the convergence of the metrics is uniform in $C^{\infty}$ on the regular sets. But we know that the tangent cone is unique, and therefore the cones $C(\Gamma_0)$ and $C(S_{P}^{\mbox{\tiny{reg}}})$ with the respective metrics must be isometric. Moreover, the convergence is in the Gromov-Hausdorff sense for pointed length spaces, and then preserve the base point. Therefore the isometry must also send the vertex of $C(\Gamma_0)$, which is the limit of $\hat{E}^{-1}(P)$, to the one of $C(S_P^{\mbox{\tiny{reg}}})$, which is the limit of $P$. As a consequence, since both $s$ and $r$ are the distances from the vertices of the cones, each slice $\{t\}\times \Gamma_0$ in $C(\Gamma_0)$ is isometric to the slice $\{t\}\times S_P^{\mbox{\tiny{reg}}}$. We have then shown that $\Gamma_0$ is isometric to the regular part of the tangent sphere $S_P^{\mbox{\tiny{reg}}}$. Hence if we take the metric completion $\hat{X}$ of $\Gamma_0$ with respect to $g$, $\hat{X}$ is an admissible stratified space of dimension $(n-1)$ isometric to the tangent sphere $S_P$ at $P$. We can extend $\hat{E}$ to $\hat{X}$ and get:
\begin{equation*}
\hat{E}: \left[-\frac{\pi}{2}, \frac{\pi}{2}\right]\times \hat{X} \rightarrow X. 
\end{equation*}
The image of $\hat{E}$ is a compact set in $X$ including the regular set $\Omega$ without at most two points. This latter is dense in $X$: therefore the image of $\hat{E}$ coincides with the whole $X$, $\hat{E}$ is surjective and it is the isometry we were looking for. 
\end{proof}


\section{A relation with the Yamabe problem}

We briefly recall here some of the basic notions about the Yamabe problem. Given a compact smooth Riemannian manifold $(M^n,g)$ of dimension $n \geq 3$, we define the conformal class of $g$ as:
\begin{equation*}
[g]=\left\{ \tilde{g}=e^{2u}g, u \in C^{\infty}(M)\right\}. 
\end{equation*}
The question posed by H.~Yamabe in 1960 was the following: does a metric with constant scalar curvature exist within the conformal class of a given Riemannian metric? The answer has been proven to be positive thanks to the works of N.~Trudinger, T.~Aubin and R.~Schoen. 

There is a classical variational formulation for the Yamabe problem. Consider the Hilbert-Einstein functional:
\begin{equation*}
Q(\tilde{g})=\frac{\displaystyle \int_M a_n S_{\tilde{g}}dv_{\tilde{g}}}{\vol_{\tilde{g}}(M)^{\frac{n-2}{n}}} \quad a_n=\frac{n-2}{4(n-1)},
\end{equation*}
and its infimum, the Yamabe constant:
\begin{equation*}
Y(M, [g]) = \inf_{\tilde{g}\in [g]} Q(\tilde{g}). 
\end{equation*}
If there exists a conformal metric $\tilde{g}$ attaining the Yamabe constant (and thus a critical point of the Hilbert-Einstein functional in the conformal class), then $\tilde{g}$ has constant scalar curvature and it is called a Yamabe metric. Observe that a metric of constant scalar curvature is not necessarily a Yamabe metric, since it is not necessarily minimizing. Nevertheless, M.~Obata proved in \cite{Obata71} that if $M^n$ carries an Einstein metric $g$ which satisfies $Ric_g=(n-1)g$, then $g$ is a Yamabe metric. Furthermore if there exists another conformal metric $\tilde{g}$ in the conformal class $[g]$, with constant scalar curvature and not homothetic to $g$, then $\tilde{g}$ is an Einstein metric as well and $(M^n,g)$ is isometric to the canonical sphere. The proof of this result is based on the existence of a conformal vector field $X$ on $(M^n,g)$. For other formulations of  the proof, see also Theorem IV.2 in \cite{BourguignonEzin} and Proposition 1.4 in \cite{Schoen87}. 

In \cite{ACM12} the authors studied the Yamabe problem on stratified spaces with the same variational approach as above, provided that the scalar curvature satisfies the appropriate integrability condition. They gave an existence result for a Yamabe metric which depends on a conformal invariant, called local Yamabe constant. In \cite{Mondello} we computed this latter under a geometric assumption on the links. We prove here a result analogous to the one of \cite{Obata71} for admissible stratified spaces: 

\begin{theorem}
\label{ObataYamabe}
Let $(X^n, g)$ be an admissible stratified space with Einstein metric. Then $g$ is a Yamabe metric. If there exists $\tilde{g}$ in the conformal class of $g$, not homothetic to $g$, with constant scalar curvature, then $\tilde{g}$ is an Einstein metric as well and $(X^n,g)$ is isometric to the spherical suspension of an Einstein admissible stratified space of dimension $(n-1)$. 
\end{theorem}

Observe that we have proven in \cite{Mondello} that the Sobolev inequality $\eqref{SobP}$ implies a lower bound for the Yamabe constant of an admissible stratified space, which is attained when the metric is Einstein: 

\begin{prop}
Let $(X^n,g)$ be an admissible stratified space. Then its Yamabe constant satisfies:
\begin{equation*}
Y(X,[g]) \geq \frac{n(n-2)}{4} \vol_g(X)^{\frac{2}{n}},
\end{equation*}
with equality if $g$ is an Einstein metric. 
\end{prop}

In fact, it suffices to compute $Q(g)$ for an Einstein metric to get exactly the right-hand side in the previous inequality. Therefore, we have already proven the first part of Theorem \ref{ObataYamabe}, that is an Einstein metric on an admissible stratified space is a Yamabe metric. We give here an alternative proof under the assumption that a non trivial Yamabe minimizer exists, that means, there exists a non trivial solution $u \in W^{1,2}(X) \cap L^{\infty}(X)$ to the Yamabe equation:

\begin{equation*}
\Delta_g u+a_n S_gu=a_nS_{\tilde{g}}u^{\frac{n+2}{n-2}}, \quad a_n=\frac{n-2}{4(n-1)}.
\end{equation*}

The transformation laws for the scalar curvature under conformal change (see Chapter 1, Section J in \cite{Besse}) imply that for a solution $u$ to the previous equation, the metric $\tilde{g}=u^{\frac{4}{n-2}}g$ is a Yamabe metric. Observe that assuming the existence of $\tilde{g}$ in the conformal class of $g$, not homothetic to $g$ and with constant scalar curvature, is equivalent to say that there exists a non-trivial solution $u$ to the Yamabe equation and that $\tilde{g}$ can be written as $u^{\frac{4}{n-2}}g$. Moreover, we can assume without loss of generality that the scalar curvature $S_{\tilde{g}}$ of $\tilde{g}$ is equal to $S_g$.  

We divide the proof of Theorem \ref{ObataYamabe} into two steps: first we prove that if $\tilde{g}$ is a metric conformal to $g$, not homothetic to $g$, with constant scalar curvature, then $\tilde{g}$ is an Einstein metric. This implies the existence of a conformal vector field on an admissible stratified space.  We partially follow an argument of J. Viaclovsky (see the proof of Theorem 1.3 in \cite{Viaclovsky}). We then give the alternative proof of the fact that an Einstein metric is a Yamabe metric: the main interest of the proof is that it shows the existence of an eigenfunction relative to the eigenvalue $n$. As a consequence we can conclude by applying the rigidity result of Theorem \ref{ObataSing}. 

\begin{theorem}
\label{Obata}
Let $(X,g)$ be an admissible Einstein stratified space of dimension $n$. Assume that there exists a metric $\tilde{g}$ in the conformal class of $g$, not homothetic to $g$, with constant scalar curvature. Then $\tilde{g}$ is an Einstein metric and there exists a function $\phi$ satisfying:
\begin{equation}
\label{csf}
\nabla d\phi =-\frac{\Delta_g \phi}{n}g.
\end{equation}
In particular, the vector field $X=d\phi$ is a conformal vector field such that $\mathcal{L}_X g=-2\phi g$.  
\end{theorem}

Before proving this theorem we recall some results contained in \cite{Mondello} and in \cite{Thesis} in order to deduce some further regularity on a Yamabe minimizer: we are going to show that if $u$ solves the Yamabe equation, then it belongs to the Sobolev space $W^{2,2}(X)$ and its gradient is bounded. 

\begin{prop}
Let $(X^n, g)$ be a stratified space. Let $F$ be a positive locally Lipschitz function and $u \in W^{1,2}(X)\cap L^{\infty}(X)$ be a non-negative solution to the equation $\Delta_g u=F(u)$. Assume that there exists a positive constant $c$ such that: 
\begin{equation}
\label{LaplacienGrad}
\Delta_g |du| \leq c |du|. 
\end{equation}
If for any $x$ in $X$ the first non-zero eigenvalue of the Laplacian on the tangent sphere $\lambda_1(S_x)$ is larger than or equal to $(n-1)$, then for any $\eps > 0$ the following control of the gradient away from an $\eps$-tubular neighbourhood of the singular set $\Sigma$ holds:
\begin{equation}
\label{tubEst}
\norm{du}_{\linfty{X \setminus \tub{\eps}}} \leq C \sqrt{|\ln{\eps}|}.
\end{equation}
where $C$ is a positive constant not depending on $\eps$. 
\end{prop}

This proposition is a consequence of Theorem A in \cite{ACM14} and of Moser iteration technique (see the proof of Proposition 1.15 in \cite{Thesis}). On a admissible stratified space $(X^n,g)$, the condition $\eqref{LaplacienGrad}$ is always satisfied thanks to the lower bound on the Ricci tensor and the Bochner-Lichnerowicz formula (see Proposition 2.3 in \cite{Thesis}). Furthermore, a Ricci lower bound on $(X^n,g)$ implies an analogous Ricci lower bound on each tangent sphere (Remark \ref{rem}): therefore, thanks to the Lichnerowicz singular theorem, the assumption on $\lambda_1(S_x)$ holds for any $x$ in an admissible stratified space. We can then reformulate the previous proposition as follows:

\begin{prop}
\label{reg}
Let $(X^n,g)$ an admissible stratified space and $u, F$ as in the previous statement. Then for any $\eps > 0$ the estimates $\eqref{tubEst}$ holds on the gradient $|\nabla u|$. 
\end{prop}

Under the assumptions of Theorem \ref{Obata}, there exists a metric $\tilde{g}=u^{\frac{4}{n-2}}g$ with constant scalar curvature $S_{\tilde{g}}$ equal to $S_g$, where $u$ is a non-negative solution to:
\begin{equation*}
\Delta_g u+a_n S_gu=a_nS_{g}u^{\frac{n+2}{n-2}}.
\end{equation*}
Since $S_g$ is equal to a constant, the function:  
\begin{equation*}
F(x)=(x^{\frac{4}{n-2}}-1)a_nS_g x.
\end{equation*}
is a locally Lipschitz function, and then we can apply Proposition \ref{reg} to the Yamabe minimizer $u$. Furthermore, we can deduce that the gradient of $u$ belongs to $\Sob{2}\cap L^{\infty}(X)$: this is done by means of an appropriate family of cut-off functions and with an argument that we developed in the proof of the singular Lichnerowicz theorem. 

\begin{lemma}
\label{reg1}
Let $(X,g)$ be an admissible stratified space of dimension $n$ with Einstein metric. Then the gradient $|\nabla u|$ of a solution $u$ to the Yamabe equation belongs to $L^{\infty}(X) \cap \Sob{2}$. 
\end{lemma}

We briefly sketch the key points of the proof. The details can be found in the proof of Theorem 2.5 in \cite{Thesis}. Let $u \in \Sob{2} \cap L^{\infty}(X)$ be a Yamabe minimizer and $F$ as above. On the regular set the Bochner-Lichnerowicz formula holds, and therefore we have:
\begin{equation*}
\nabla^*\nabla du +Ric_g(du) = F'(u)du \mbox{ on } \Omega. 
\end{equation*}
Since $u$ is bounded and $F$ is locally Lipschitz, there exists a positive constant $c$ such that $|F'(u)|\leq c$. Moreover, thanks to the fact that $Ric_g=(n-1)g$ there exists a positive constant $c_1$ such that:
\begin{equation*}
\frac 12 \Delta_g (|\nabla u|^2)=(\nabla^*\nabla du, du)-|\nabla du|^2 \leq c_1 |\nabla u|^2 -|\nabla du|^2. 
\end{equation*}
In the proof of the singular Lichnerowicz theorem in \cite{Mondello} we defined a family of cut-off functions $\cutoff$, $0 \leq \cutoff \leq 1$, being equal to one outside a tubular neighbourhood $\Sigma^{2\eps}$ of the singular set, vanishing on $\Sigma^{\eps}$. Furthermore, if the estimate $\eqref{tubEst}$ holds  for $u$, then the cut-off functions $\cutoff$ are constructed in such a way that they satisfy: 
\begin{equation*}
\lim_{\eps \rightarrow 0}\int_X (\nabla u,\nabla \cutoff)dv_g=0, \quad \lim_{\eps \rightarrow 0} \int_X \Delta_g(\cutoff) |\nabla u|^2dv_g=0. 
\end{equation*}
If we multiply the previous inequality by $\cutoff$ and integrate by parts we obtain: 
\begin{equation}
\label{gradientYamabe}
\frac 12 \int_X (\Delta_g \cutoff) |\nabla u|^2dv_g \leq c_1 \int_X \cutoff |\nabla u|^2 dv_g -\int_X \cutoff |\nabla du|^2 dv_g. 
\end{equation}
The left-hand side of $\eqref{gradientYamabe}$ tends to zero as $\eps$ tends to 0: as a consequence, the norm in $L^2(X)$ of $\nabla du$ is bounded by the one of $|\nabla u|$, which is finite. This means that $\nabla |\nabla u|$ belongs to $L^2(X)$, and $|\nabla u|$ to $W^{1,2}(X)$. 

We also know that on the regular set $\Omega$ the inequality $\Delta_g |\nabla u| \leq c_1 |\nabla u|$ is satisfied. One can prove that this implies the weak inequality $\Delta_g |\nabla u|\leq c_1 |\nabla u|$ on the whole $X$, again by integrating by parts and by using the cut-off functions $\cutoff$. Finally, a positive function $f$ in $W^{1,2}(X)$ satisfying the weak inequality $\Delta_g f \leq c_1 f$ on $X$ belongs to $L^{\infty}(X)$, thanks to the Moser iteration technique (see Proposition 1.8 in \cite{ACM12}). Therefore $u$ belongs to $W^{2,2}(X)$ and its gradient $|\nabla u|$ is bounded. \newline 

We are now in position to prove Theorem \ref{Obata}.

\begin{proof}[Proof of Theorem \ref{Obata}]
By assumption, there exists a conformal metric $\tilde{g}\in[g]$ with constant scalar curvature: we can assume without loss of generality that $S_{\tilde{g}}=S_g=n(n-1)$. Since $\tilde{g}$ is not homothetic to $g$, there exists a function $u \in \Sob{2}{X}\cap L^{\infty}(X)$ solving the Yamabe equation:
\begin{equation*}
\Delta_g u +\frac{n(n-2)}{4}u = \frac{n(n-2)}{4} u^{\frac{n+2}{n-2}}
\end{equation*}
and such that $\tilde{g}=u^{\frac{4}{n-2}}g$. In order to simplify the transformation formulas under conformal change, we can define $\phi=u^{-\frac{2}{n-2}}$, so that $\tilde{g}=\phi^{-2}g$. We know from Theorem 1.12 in \cite{ACM12} that $u$, and thus $\phi$, is positive and bounded. By the previous discussion we also have that the gradient of $u$, and therefore the gradient of $\phi$, belongs to $L^{\infty}(X)$. 

Consider the traceless Ricci tensor $E_{\tilde{g}}$ and recall that $\tilde{g}$ is an Einstein metric if and only if $E_{\tilde{g}}$ vanishes: our goal is to show that this is the case. The transformation law for the traceless Ricci tensor under a conformal change (see for example \cite{Besse}) gives us the following formula for $E_{\tilde{g}}$:
\begin{equation*}
E_{\tilde{g}}= E_g+(n-2)\phi^{-1}\Big(\nabla^2\phi + \frac{\Delta_g\phi}{n}g\Big)
\end{equation*}
where the covariant derivatives are taken with respect to $g$. Since by assumption $g$ is an Einstein metric, $E_g=0$.
Then consider the following integral:
\begin{equation*}
I_{\eps}=\int_X \cutoff \phi |E_{\tilde{g}}|_g^2dv_g
\end{equation*}
where $\cutoff$ is chosen like in the proof of Lemma \ref{reg1}. If we show that $I_{\eps}$ tends to zero as $\eps$ goes to zero, then the norm of $E_{\tilde{g}}$ must vanish: as a consequence we will obtain that $\tilde{g}$ is an Einstein metric and that its conformal factor $\phi$ satisfies $\eqref{csf}$. Let us rewrite $I_{\eps}$ in the appropriate form:
\begin{align*}
I_{\eps} 
& = \int_X \cutoff \phi\left(E_{\tilde{g}}, (n-2)\phi^{-1}\left( \nabla d\phi + \frac{\Delta_g \phi}{n} g\right) \right)_g dv_g \\
 & = (n-2)\int_X \cutoff \left( E_{\tilde{g}}, \nabla d\phi + \frac{\Delta_g \phi}{n} g\right)_g dv_g \\
 & = (n-2) \int_X \cutoff \left( E_{\tilde{g}}, \nabla d\phi\right)_g dv_g.
\end{align*}
Then we integrate by parts:
\begin{equation*}
\int_X \cutoff \left( E_{\tilde{g}}, \nabla d\phi\right)_g dv_g = \int_X  (E_{\tilde{g}}^{ij}\nabla_j \cutoff \nabla_i \phi + \cutoff \nabla_j E_{\tilde{g}}^{ij} \nabla_i \phi) dv_g. 
\end{equation*}
Since the scalar curvature of $\tilde{g}$ is constant, by the Bianchi identity (see also \cite{BourguignonEzin}), which holds on the regular set of $X$, the second term of this integral is equal to zero. The first one leads to:
\begin{equation}
\label{Ieps}
I_{\eps}=(n-2)^2\int_X \phi^{-1}\left( \nabla d\phi (\nabla \cutoff, \nabla \phi) + \frac{\Delta_g \phi}{n}(\nabla \cutoff, \nabla \phi)_g\right) dv_g. 
\end{equation}
Observe that $\phi^{-1}$ is positive and bounded, because the solution $u$ to the Yamabe equation is positive and bounded thanks to Theorem 1.12 in \cite{ACM12}. We claim that the Laplacian of $\phi$ is bounded as well. In fact, if we denote $p=-\frac{2}{n-2}$ we have:
\begin{equation*}
\Delta_g \phi= p u^{p-1}\left( \Delta_g u - (p-1)\frac{|\nabla u|^2}{u}\right). 
\end{equation*}
As we recalled above, the function $u$ is bounded and positive, then its Laplacian $\Delta_g u$ is bounded, since it is equal to:
\begin{equation*}
\Delta_g u= \frac{n(n-2)}{4}u(u^{\frac{4}{n-2}}-1). 
\end{equation*}
Moreover, by the previous Lemma the gradient $|\nabla u|$ belongs to $L^{\infty}(X)$, so that the same holds for $\Delta_g \phi$. Therefore, if we consider the last term in $\eqref{Ieps}$, we know that $\cutoff$ is chosen in such a way that the integral of $(\nabla \cutoff, \nabla u)$ goes to zero as $\eps$ tends to zero.

As for the first term in $\eqref{Ieps}$, we can integrate by parts and obtain:
\begin{equation*}
\int_X  \nabla d\phi (\nabla \cutoff, \nabla \phi) dv_g =\frac{1}{2} \int_X \cutoff \Delta_g |\nabla \phi|^2 dv_g = \frac 12 \int_X (\Delta_g \cutoff)|\nabla \phi|^2 dv_g. 
\end{equation*}
The cut-off functions $\cutoff$ are chosen in such a way that this last term tends to zero as $\eps$ goes to zero as well. 

As a consequence, we have shown that $I_{\eps}$ tends to zero as $\eps$ goes to zero. Therefore we obtain that the norm of the traceless Ricci tensor $E_{\tilde{g}}$ is equal to zero, the metric $\tilde{g}$ is an Einstein metric and the function $\phi$ satisfies $\eqref{csf}$, as we wished. 
\end{proof}

A scalar function solving the equation $\eqref{csf}$ is called in the literature a concircular scalar field. The existence of a concircular scalar field or of a conformal vector field on a compact, or complete, smooth manifold can lead to various consequences. For example, Y.~Tashiro in \cite{Tashiro} classified complete manifolds possessing a concircular scalar field. See also Sections 2 and 3 of \cite{Montiel} for a brief but complete presentation of some known results about the subject. 

In our case, the previous theorem leads to the following:

\begin{cor}
\label{YamabeSimpleEdges}
Let $(X,g)$ be an admissible Einstein stratified space of dimension $n$ admitting a Yamabe minimizer
\begin{equation*}
\tilde{g}=\phi^{-2/(n-2)}g
\end{equation*}
Assume that $\phi$ is not a constant function. Then the Einstein metric $g$ attains the Yamabe constant, which is consequently equal to
\begin{equation*}
Y(X,[g])= \frac{n(n-2)}{4}\vol_g(X)^{\frac{2}{n}}.
\end{equation*}
\end{cor}

\begin{proof}
We have proven in the previous theorem that any metric with constant scalar curvature in the conformal class of $g$ is an Einstein metric and it is determined by a positive solution of $\eqref{csf}$. Up to multiplying by a constant, a positive solution of $\eqref{csf}$ is given by 
\begin{equation*}
\phi_t=(1-t)\phi+t 
\end{equation*}
for some $t \in [0,1)$. Let us denote:
\begin{equation*}
u_t=\phi^{-\frac{2n}{n-2}}_t.
\end{equation*}
the corresponding solution to the Yamabe equation. The metric $g_t=\phi_t^{-2}g$ is still an Einstein metric in the conformal class of $g$ and has the same scalar curvature as $g$.

We want to show that the volume of $X$ with respect to the metric $g_t$ is constant in $t$: this means that it is constant among the metrics with constant scalar curvature equal to $n(n-1)$. In this way, the ratio
\begin{equation*}
\displaystyle Q(\tilde{g})=\frac{\displaystyle a_n \int_X\mbox{Scal}_{\tilde{g}}dV_{\tilde{g}}}{\mbox{Vol}_{\tilde{g}}(X)^{1-\frac{2}{n}}}
\end{equation*}
does not decrease in the set of conformal metrics with constant scalar curvature. As a consequence, the Yamabe constant of $(X,g)$ will be attained by $g$ and it is equal to:
\begin{equation*}
Y(X,[g])= \frac{n(n-2)}{4}\vol_g(X)^{\frac{2}{n}}.
\end{equation*}
The volume of $X$ with respect to $g_t$ is given by the formula
\begin{equation*}
\mbox{Vol}_{g_t}(X)=\int_X u_t^{\frac{2n}{n-2}}dv_g=\int_X dv_{g_t}.
\end{equation*}
where we denote with $dv_{g_{t}}$ the volume element with respect to $g_t$. 
If we differentiate with respect to $t$ we get
\begin{equation}
\label{der}
\frac{\mbox{d}}{\mbox{dt}}\mbox{Vol}_{g_t}(X)=\frac{2n}{n-2}\int_X u_t^{\frac{n+2}{n-2}}\dot{u}_tdv_g=\frac{2n}{n-2}\int_X \frac{\dot{u}_t}{u_t}dv_{g_t}.
\end{equation}
We are going to show that this integral is equal to zero. If we set
\begin{align*}
v_h &=\frac{u_{t+h}}{u_t}. \\
g_h & =v_h^{\frac{4}{n-2}}g_t=u_{t+h}^{\frac{4}{n-2}}g.
\end{align*}
then $v_h$ satisfies the Yamabe equation with respect to $g_t$:
\begin{equation*}
\Delta_{g_t}v_h+\frac{n(n-2)}{4}v_h=\frac{n(n-2)}{4}v_h^{\frac{n+2}{n-2}}.
\end{equation*}
By deriving this equality with respect to $h$ we obtain
\begin{equation*}
\Delta_{g_t}\dot{v}_h+\frac{n(n-2)}{4}\dot{v}_h=\frac{n(n+2)}{4}v_h^{\frac{4}{n-2}}\dot{v}_h.
\end{equation*}
and when $h=0$ we have as a consequence $\Delta_{g_t}\dot{v}_0=n\dot{v}_0$, that is $v_0$ is an eigenfunction relative to the first eigenvalue $n$ of $\Delta_{g_t}$. 
Any eigenfunction relative to the first eigenvalue has mean equal to zero over $X$, so that we have: 
\begin{equation*}
\int_X \dot{v}_0 dV_{g_t}=0.
\end{equation*}
But by definition $\dot{v}_0$ is equal to $\displaystyle \frac{\dot{u}_t}{u_t}$. Recalling $\eqref{der}$ we have obtained that the volume of $X$ is constant with respect to $t$: this implies that the Einstein metric $g$ attains the Yamabe constant, as we wished.
\end{proof}

\begin{cor}
\label{rigidité}
Let $(X^n, g)$ be an Einstein admissible stratified space. If there exists $\tilde{g}$ in the conformal class of $g$, not homothetic to $g$, with constant scalar curvature, then $(X^n,g)$ is isometric to the spherical suspension of an Einstein admissible stratified space of dimension $(n-1)$.
\end{cor}

\begin{proof}
The proof of the previous Corollary implies that there exists an eigenfunction $\dot{v}_0$ associated to the eigenvalue $n$, therefore we can apply Theorem \ref{ObataSing}. 
\end{proof}

If we collect Theorem \ref{Obata} and Corollaries \ref{YamabeSimpleEdges} and \ref{rigidité}, we have proven Theorem \ref{ObataYamabe}. 

\bibliography{BiblioObata.bib}

\end{document}